\documentclass[twoside,12pt]{amsart}
\usepackage{amsmath,latexsym,amssymb,mathptm, times,verbatim, enumerate,mathcomp,lscape}
    \usepackage{longtable}\usepackage{tikz,stmaryrd}
\usepackage{pdfpages}
\input amssym.def
\input amssym
\input xypic
\input xy
\xyoption{all}
\usepackage[right=1.5in,left=1.5in,top=1in,bottom=1in]{geometry}

\newcommand{\agr}[2][{}]{{{#2}^{\mathsf g}_{#1}}}

\def\HS{\operatorname{HS}}

\def\im{\operatorname{im}}
\def\mult{\operatorname{mult}}
\def\m{\mathfrak m}
\def\n{\mathfrak n}
 
\def\ann{\operatorname{ann}}
\def\socle{\operatorname{socle}}

\def\Hom{\operatorname{Hom}}
\def\kk {\pmb k}
\def\Tor{\operatorname{Tor}}
\def\t{\otimes}

\def\HH{\operatorname{H}}

\def\pp{\mathfrak p}
\def\incl{\operatorname{incl}}

\newtheorem{theorem}{Theorem}[section]
\newtheorem{lemma}[theorem]{Lemma}
\newtheorem{corollary}[theorem]{Corollary}
\newtheorem{proposition}[theorem]{Proposition}
\newtheorem{proposition-no-advance}[equation]{Proposition}

\newtheorem{claim-no-advance}[equation]{Claim}
\newtheorem{observation}[theorem]{Observation}

\newtheorem{quick consequences}[theorem]{Quick Consequences}

\theoremstyle{definition}
\newtheorem{facts and definitions}[theorem]{Facts and Definitions}
\newtheorem{definition}[theorem]{Definition}
\newtheorem{remark}[theorem]{Remark}
\newtheorem{remark-no-advance}[equation]{Remark}

\newtheorem{data}[theorem]{Data}
\newtheorem{setup}[theorem]{Set up}

\newtheorem{careful calculation}[theorem]{Careful Calculation}

\newtheorem{present summary}[theorem]{Present Summary}

\newtheorem{further reductions}[theorem]{Further Reductions}

\newtheorem{chunk}[theorem]{}
\newtheorem{chunk-no-advance}[equation]{}

\newtheorem{marching orders}[theorem]{Marching Orders}

\newtheorem{circle the wagons}[theorem]{Circle the wagons}

\newtheorem*{Remark}{Remark}

\numberwithin{equation}{theorem}
\numberwithin{table}{theorem}

\begin{document}

\baselineskip=16pt

\title [Poincar\'e series of  compressed  local Artinian rings]{Poincar\'e series of compressed  local Artinian rings with odd top socle degree}

\date{\today}

\author[A.~R.~Kustin]{Andrew R.~Kustin}
\address{Andrew R.~Kustin\\ Department of Mathematics\\ University of South Carolina\\\newline 
Columbia\\ SC 29208\\ U.S.A.} \email{kustin@math.sc.edu}

\author[L.~M.~\c{S}ega]{Liana M.~\c{S}ega}
\address{Liana M.~\c{S}ega\\ Department of Mathematics and Statistics\\
   University of Missouri\\  Kansas City\\ MO 64110\\ U.S.A.}
     \email{segal@umkc.edu}

\author[A.~Vraciu]{Adela~Vraciu}

\address{Adela~Vraciu\\ Department of Mathematics\\ University of South Carolina\\
Columbia\\ SC 29208\\ U.S.A.} \email{vraciu@math.sc.edu}

\subjclass[2010]{13D02, 13D40, 13D07, 13E10, 13A02, 16E45}

\keywords{compressed ring, differential graded algebra, generic algebra, Golod homomorphism, Grassmannian, homology algebra of the Koszul complex, Koszul homology,   Poincar\'e series, socle, trivial Massey operation.}

\thanks{Research partly supported by  Simons Foundation collaboration grant 233597 (ARK), Simons Foundation collaboration grant 354594 (LM\c{S}), NSF grant  DMS-1200085 (AV)}

\thanks{Macaulay2 \cite{M2} was very valuable to us. We used it to test various conjectures.}

\begin{abstract}  We define a notion of compressed local Artinian ring that does not require the ring to contain a field. Let $(R,\m)$ be  a compressed  local Artinian ring
with odd top socle degree $s$, at least five, and $\socle(R)\cap \m^{s-1}=\m^s$.  We prove that 
the Poincar\'e series of all finitely generated modules over $R$ 
are 
rational, sharing a common denominator,
and that there is a Golod homomorphism from a complete intersection onto $R$. 
\end{abstract}

\maketitle

\section{Introduction.}

Let $(R,\mathfrak m,\kk)$ be a local ring and $M$ be a finitely generated $R$-module. The Poincar\'e series of $M$, $$P^R_M(z)=\sum\limits_{i=0}^\infty b_i(M) z^i,$$  is the generating function for the sequence of Betti numbers of $M$:
\begin{align*}b_i(M)&{}=\text{ the minimal number of generators of the $i^{\text{th}}$ syzygy of $M$}\\&{}=\dim_{\kk}\Tor_i^R(M,\kk).\end{align*} In the 1950's Kaplansky and Serre \cite[pg.~118]{Se65}  asked if the Poincar\'e series of a local ring is always a rational function. Considerable study was devoted to this question, 
(see, for example, the survey articles \cite{Ro,Ba}), before Anick \cite{An2} showed that the answer is no.   

Consideration of rational and transcendental Poincar\'e series has only intensified since the appearance of Anick's example. The example has been simplified, reworked, and reformulated in the language of Algebraic Topology; see the discussion following Problem 4.3.10 in \cite{A98} for more details, including references.  
Roos \cite{Ro05} calls a local ring $R$ {\it good} if the Poincar\'e series of all finitely generated modules over $R$ 
are 
rational, sharing a common denominator. 
A list of applications of the hypothesis that a local ring is good may be found in
\cite{Av94}. 
The recent papers \cite{HR,EV,CN,CENR}
all prove that a family of rings has rational 
Poincar\'e series. 

Nonetheless,  
at the Introductory
Workshop  for the special year in Commutative Algebra at the Mathematical Sciences Research Institute in 2012 
Irena Peeva  observed \cite{MP} that 
``We do not have a feel for which of the following cases holds.
\begin{enumerate}[\rm(a)]
\item Most Poincar\'e series are rational, and irrational Poincar\'e series occur rarely
in specially crafted examples.
\item  Most Poincar\'e series are irrational, and there are some nice classes of rings
(for example, Golod rings, complete intersections) where we have rationality.
\item  Both rational and irrational Poincar\'e series occur widely.
\end{enumerate}
One would like to have results showing whether the Poincar\'e series are rational
generically, or are irrational generically.''

 A first answer to Peeva's problem is made in the paper by Rossi and \c Sega, \cite{RS}, where it is shown that if $R$ is a compressed Artinian Gorenstein local ring with top socle degree not equal to three, then 
the Poincar\'e series of all finitely generated modules over $R$ 
are 
rational, sharing a common denominator. (In particular, these rings are good, in the sense of Roos.)  
The Rossi-\c{S}ega theorem is a complete answer to the Peeva problem for generic Artinian Gorenstein rings because generic Artinian Gorenstein rings are automatically compressed.  
Furthermore, it is necessary to avoid top socle degree three because B{\o}gvad \cite{B83} has given  examples of compressed Artinian Gorenstein rings with top socle degree three which have transcendental Poincar\'e series. 

In the present paper we carry the Rossi-\c{S}ega program further.
As in the Gorenstein case, once the relevant parameters are fixed, the set of
 Artinian standard-graded $\kk$-algebras is parameterized (now by a non-empty open subset in a chain of relative Grassmannians) and (when $\kk$ is infinite) the points on a non-empty open subset of  this parameter space 
 correspond to  compressed algebras. As in \cite{RS}, we ignore the parameter space 
 and the cardinality of $\kk$; instead, we prove 
that 
compressed  
local Artinian rings, 
with odd top socle degree $s$  are also good in the sense of Roos,
provided $5\le s$ and $\socle(R)\cap \m^{s-1}=\m^s$. Our result then applies in the ``generic case'' whenever the ``generic case'' makes sense.
B{\o}gvad's examples also apply in our situation, so we also are forced to exclude top socle degree equal to three.

Our argument is inspired by  the proof in \cite{RS}. The key ingredient from local algebra  
in our proof is Lemma~\ref{FirstStep} which can be interpreted as a statement about the  structure of the  Koszul homology algebra
$$\HH_\bullet(R\t_Q K)=\Tor_\bullet^Q(R,\kk),$$
where $Q\to R$ is a    surjection of local rings 
of the same embedding dimension $e$, $(Q,\n,\kk)$ is a regular local ring, $(R,\m,\kk)$ is
a compressed 
 local Artinian ring,
and $K$ is the Koszul complex which is a minimal  
resolution of $\kk$ by free $Q$-modules. When the hypotheses of Lemma~\ref{FirstStep} are in effect, then the conclusion may be interpreted to say that there is an element $\bar g$ in $\Tor_1^Q(R,\kk)$ with
$$\bar g\cdot \Tor_{e-1}^Q(R,\kk)=\Tor_{e}^Q(R,\kk).$$We use this conclusion to create a Golod homomorphism from a hypersurface ring onto $R$.

Our reliance on Lemma~\ref{FirstStep} explains the hypotheses in the main theorem
about  the shape of the socle of $R$. In particular, 
when the top socle degree of $R$ is even, it is possible for a non-Golod compressed   Artinian standard-graded $\kk$-algebra $R$ to have $\Tor_1\cdot \Tor_{e-1}=0$. (We are writing $\Tor_i$ in place of $\Tor_i^Q(R,\kk)$.) For example, if $e=4$ and  $\socle(R)$ is  isomorphic to $\kk(-4)^2$, then according to \cite[Conj.~3.13 and 4.1.2]{B99}, the Betti table 
for the minimal homogeneous resolution of $R$ by free $Q$-modules is 
$$\begin{matrix}            &0&  1 & 2&  3& 4\\
\text{total}: &1 &12 &19 &10& 2\\
         0: &1 & . & . & .& .\\
         1: &. & . & . & .& .\\
         2: &. &12 &15 & .& .\\
         3: &. & . & 4 &10& .\\
         4:& . & .&  . & .& 2\end{matrix} \quad \text{or}\quad\left(\begin{matrix} 12&15&0\\0&4&10\end{matrix}\right),$$ in the language of Macaulay2 \cite{M2} or Boij \cite[Notation~3.4]{B99},  respectively.
The numerology alone shows that $\Tor_1\cdot \Tor_3=0$, but the numerology permits $\Tor_2\cdot \Tor_2$ to be non-zero, and this is precisely what happens. In a similar manner, 
if the top socle degree of $R$ is  three or if  $(\socle(R)\cap \m^{s-1})/\m^s\neq 0$, then the Betti tables (in the homogeneous case) 
 permit too many non-zero products in $\Tor$. Consequently, if $R$ is compressed and  the top socle degree of $R$ is $3$, or  the top socle degree of $R$ is even, or  $(\socle(R)\cap \m^{s-1})/\m^s$ is non-zero, then our techniques are not able to determine  if the Poincar\'e series of $R$ is rational. In these cases, the question of Peeva remains wide open. 

We prove that the Poincar\'e series of $R$ is rational by exhibiting a Golod surjection from a complete intersection onto $R$. It is worth observing that the existence of such a map is an important conclusion in its own right. For example, 
this hypothesis is used in \cite{L16} in the study of the rigidity of the two-step Tate complex, in \cite{AINS} in study of the non-vanishing of $\Tor^R_i(M,N)$  for infinitely many $i$, and in \cite{AISS} in the study of the structure of the set of semi-dualizing modules of a ring $R$.

Theorem~\ref{5.1} is the main result of the paper. To prove this theorem  we
apply Lemma~\ref{RS1.2}, which is established in \cite{RS}. Lemma~\ref{RS1.2} is a 
down-to-earth
criteria for proving that a given surjection of local rings is a Golod homomorphism.  Massey operations are replaced with calculations involving $\Tor^Q_{\bullet}(-,\kk)$, where $Q$ is a regular local ring.

A compressed 
 local Artinian ring $R$ exhibits extremal behavior. Such a ring  has maximal length among all 
 local Artinian rings with the same embedding dimension and  socle polynomial. 
Extremal objects  
 exhibit special properties and deserve extra study. Indeed,
there are many applications of compressed rings and these rings have received much  study;  see, for example, \cite{S80,I84, FL84, F87, B92, GP98, B99, B00, MM03, Z03, Z04, MMN05, K07, C10, D14, RS, ER15, HS17}.
However, for  thirty years, 1984\,--\,2014,  the notion of ``compressed'' ring 
 was only defined for rings containing a field. 
Finally, in 2014, Rossi and \c Sega \cite{RS} proved that  
the notion of 
``compressed  local Artinian \underline{Gorenstein}  ring'' is meaningful,  interesting, and works just as well in the non-equicharacteristic case. 
Furthermore, their theorem about rational Poincar\'e series is valid in the non-equicharacteristic case.
 
In sections
\ref{mult} and \ref{comp} we embrace the philosophy of \cite{RS} and prove that the phrase ``compressed  local Artinian   ring'' is meaningful whether or not the ring contains a field and whether or not the ring is Gorenstein. Furthermore our main theorem, Theorem~\ref{5.1}, is valid in the context of this enlarged notion of ``compressed ring''. It is worth noting that although we adopt the philosophy of \cite{RS}, the techniques about Gorenstein compressed rings in \cite{RS} are not relevant in our situation.     Our technique is introduced in Section~\ref{mult} and the proof that the phrase 
``compressed  local Artinian   ring'' is meaningful is carried out in Section~\ref{comp}. This study of compressed local Artinian rings is an important feature of the present paper.

Section~\ref{Prelims} consists of preliminary matters. 
In Section~\ref{mult} we introduce our duality technique for studying  local Artinian   rings. In Section~\ref{comp} we prove that the notion of compressed ring is meaningful in the non-equicharacteristic case. Section~\ref{GH} is concerned with Golod homomorphisms and the homological algebra that can be used to prove that a homomorphism is Golod. In Section~\ref{AIL} we explore the consequences of the hypothesis ``compressed'' on the homological algebra of Section~\ref{GH}. The proof of the main theorem is given in Section~\ref{mainResult}. In Section~\ref{consolation} it is shown that in the situation of the main theorem, $R/\m^s$ is a Golod ring and the natural map $R\to R/\m^s$ is a Golod homomorphism; furthermore, the final statement continues to hold even if $(\socle(R)\cap \m^{s-1})/\m^s$ is not zero.

\section{Notation, conventions, and preliminary results.}\label{Prelims}

In this paper ${\kk} $ is always a field. 
\begin{chunk}Let $I$ be an ideal in a ring $A$, $N$ be an $A$-module, and  $L$ and $M$ be submodules of $N$. Then 
$$L:_IM=\{x\in I\mid xM\subseteq L\}\quad\text{and}\quad L:_MI=\{m\in M\mid Im\subseteq L\}.$$If $L$ is the zero module, then we also use ``annihilator notation'' to describe these ``colon modules''; that is,
$$\ann_AM=0:_AM\quad\text{and}\quad \ann_NI=0:_NI.$$Any undecorated ``$:$'' or ``$\ann$'' means $:_A$ or $\ann_A$, respectively, where $A$ is the ambient ring.  \end{chunk}

\begin{chunk}\label{N27.1} If $I$ is an ideal in a  ring $A$, $N$ is an $A$-module, and $L$ and $M$ are submodules of $N$ with $IL\subseteq M$, then  let 
$$\mult:I\to \Hom_A(L,M)$$ denote the homomorphism which sends the element $\theta$ of $I$ to the homomorphism $\mult_{\theta}$ of $\Hom_A(L,M)$, where $\mult_{\theta}(\ell)=\theta \ell$ for all $\ell$ in $L$. \end{chunk}

\begin{chunk}\label{2.2} 
``Let $(R,\mathfrak m,\kk)$ be a local ring'' identifies $\mathfrak m$ as the unique maximal ideal of the commutative Noetherian local ring $R$ and $\kk$ as the residue class field $\kk=R/\mathfrak m$. \begin{enumerate}[\rm(a)] 
\item The {\it embedding dimension} of $R$ is $e=\dim_{\kk} (\m/\m^{2})$. 
\item The {\it Hilbert function} of $R$ is the function
$$i \mapsto h_R(i)=\dim_{\kk} (\m^i/\m^{i+1}).$$
\item If $M$ is an $R$-module, then the {\it socle} of $M$ is the vector space $\socle(M)=0:_M\mathfrak m$.
\item\label{2.2.d} If $(R,\mathfrak m,\kk)$ is  a local Artinian ring, then the {\it top socle degree} of $R$ is the maximum integer $s$ with $\m^s\neq 0$ and the {\it socle polynomial} of $R$ is the formal polynomial $\sum_{i=0}^sc_i z^i$, where
$$c_i=\dim_{\kk}\frac{\socle(R)\cap \m^i}{\socle(R)\cap \m^{i+1}}.$$
Further comments about the phrase ``top socle degree'' may be found in Remark~\ref{2.10}.
\item If $M$ is a finitely generated $R$-module, then $\mu(M)$ denotes the minimal number of generators of $M$.
\item\label{v(R)} The parameter ${\tt v}(R)$ is defined by $$\textstyle {\tt v}(R)=\inf\left\{i\left| \dim_{\kk} (\m^i/\m^{i+1}) < \binom{(e-1)+i}{i}\right.\right\},$$where $e$ is the embedding dimension of $R$. (This notation  is introduced in \cite[(4.1.1)]{RS}.)
\end{enumerate}
\end{chunk}
\begin{Remark} Let $(R,\m,\kk)$ be a local Artinian ring with top socle degree $s$ and socle polynomial equal to $\sum_{i=0}^sc_iz^i$. Part of the hypothesis of Theorem~\ref{5.1} is that $\socle(R)\cap \m^{s-1}=\m^s$. This condition is equivalent to $c_{s-1}=0$. \end{Remark}

\bigskip
Observation~\ref{obs} follows quickly from the definition of ${\tt v}(R)$, gives an idea of the significance of this invariant, and is used in the proof of Lemma~\ref{large powers}.
\begin{observation}\label{obs}If $(R,\m,\kk)$ is a local 
 Artinian  ring, $x_1$ is a minimal generator of $\m$,  and $i$ is an integer with $0\le i\le {\tt v}(R)-2$, 
then the linear transformation 
 \begin{equation}\label{name-me}x_1:\m^i/\m^{i+1}\to \m^{i+1}/\m^{i+2},\end{equation}
which is given by multiplication by $x_1$, is an injection.  In particular, 
$$(\m^j:\m)=\m^{j-1},\quad \text{for $1\le j\le {\tt v}(R)$,}$$and $\socle(R)\subseteq \m^{{\tt v}(R)-1}$.
\end{observation}

\begin{proof}Extend the set $\{x_1\}$ to be a minimal generating set 
  $x_1,x_2,\dots,x_e$ 
for $\m$. If $d$ is an arbitrary non-negative integer, then the  set of monomials in $x_1,\dots,x_e$ of degree $d$ represents a generating set for $\m^d/\m^{d+1}$.  If $d<{\tt v}(R)$, then the number of monomials in this set is equal to the dimension of the vector space $\m^{d}/\m^{d+1}$ and hence this set of monomials is a basis for the vector space.  The index $i$  satisfies $i+1<{\tt v}(R)$; consequently, the linear transformation (\ref{name-me}) carries a basis of $\m^i/\m^{i+1}$ to  part of a basis of  $\m^{i+1}/\m^{i+2}$; and  therefore, this linear transformation 
is an injection. \end{proof}

\begin{definition}\label{cmped-new*} Let $(R,\m,\kk)$ be a local Artinian ring of 
embedding dimension $e$, top socle degree $s$, and socle polynomial $\sum_{i=0}^s c_iz^i$. 
If   
 the   Hilbert function of $R$ is given by
$$\textstyle 
\dim_{\kk} (\m^i/\m^{i+1}) = \min \Big\{ \binom{(e-1)+i}{i}, \ 
\sum\limits_{\ell=i}^s c_\ell\binom{(e-1)+(\ell-i)}{\ell-i}
\Big\},\quad \text{for $0\le i\le s$,}$$ then $R$ is called a 
{\it compressed 
 local Artinian ring}.
\end{definition}
Alternate definitions of ``compressed  local Artinian ring'' are given in Theorem~\ref{comp-ring} and Remark~\ref{4.4.2}.

\begin{chunk}If $S$ is a ring and $M$ is an $S$-module, then let $\lambda_S(M)$ denote the length of $M$ as an $S$-module.\end{chunk}

\begin{chunk}
Let $\kk$ be a field. A graded ring $R=\bigoplus_{0\le i}R_i$ is called a {\it standard-graded  ${\kk} $-algebra}, if $R_0={\kk} $, $R$ is generated as an $R_0$-algebra by $R_1$, and $R_1$ is finitely generated as an $R_0$-module.  \end{chunk}

\begin{chunk}\label{agr} If $M$ is a module over the local ring $(R,\m,\kk)$, then 
$$\agr{M}=\bigoplus_{i=0}^\infty \m^iM/\m^{i+1}M\quad \text{and}\quad \agr{R}=\bigoplus_{i=0}^{\infty} \m^i/\m^{i+1}$$ are the associated graded objects with respect the the maximal ideal: $\agr{R}$ is a standard graded $\kk$-algebra and $\agr{M}$ is a graded $R$-module. 
\end{chunk}

\begin{chunk} If $V$ is a graded vector space over the field $\kk$ with $V_i$ finite dimensional for all $i$ and $V_i=0$ all sufficiently small $i$, then the formal Laurent series $$\HS_V(z)=\sum_{i} \dim_{\kk}(V_i)\, z^i$$
 is called the {\em Hilbert series} of $V$.\end{chunk}

\begin{remark}\label{2.10} Let $(R,\mathfrak m,\kk)$ be  a local Artinian ring. There are various reasons that the expression ``top socle degree'', which is introduced in \ref{2.2}.(\ref{2.2.d}), is appropriate.
\begin{enumerate}[\rm(a)]
\item This is the expression that Iarrobino used when he defined the notion of  compressed algebra in \cite[Def.~2.2]{I84}.
\item The top socle degree of $R$ 
 is the {\bf degree} of the {\bf socle} polynomial of $R$.
\item The Hilbert series of the associated graded ring $\agr{R}$ of $R$ is a polynomial of {\bf degree} equal to the top socle degree of $R$.
\item The top socle degree of $R$  
is the {\bf top degree} of the associated graded ring $\agr{R}$.
\end{enumerate}\end{remark}

\bigskip The following calculation  is used in the proof of Lemma~\ref{4.2.c}.
\begin{remark}\label{Rmk11} If $\kk$ is an infinite field, $(Q,\n,\kk)$ is a regular 
local ring 
of embedding dimension $e$, $t$ is an integer, 
 and $h_0$ is an element of $\n^t\setminus \n^{t+1}$, then there exists 
a minimal generating set  $X_1,\dots,X_e$ for $\n$ such that $h_0-uX_1^t$  is in the ideal $(X_2,\dots,X_e)\mathfrak n^{t-1}$, for some unit $u$ in $Q$. In particular, there is a generator $h$ for the ideal $(h_0)$ of $Q$ such that $h-X_1^t\in (X_2,\dots,X_e)\mathfrak n^{t-1}$.
\end{remark}

\medskip\noindent{\it Outline of proof.}
Pass to the associated graded ring $\agr{Q}$.
If $h_0$ is a non-zero homogeneous form of degree $t$ in $\kk[X_1,\dots,X_e]$, where $\kk$ is an infinite field, then, there exists  a homogeneous change of variables $$X_1\mapsto x_1,\quad \text{and}\quad  X_i\mapsto a_ix_1+x_i,\text{ for $2\le i\le e$},$$ such that, in the new variables, $h=x_1^t+g$, where $g$ is a homogeneous form of degree $t$ in the ideal $(x_2,\dots,x_e)$.
The proof is clear.  Start with $$h_0=\sum \limits_{j_1+\dots+j_{e}=t}A_{j_1,\dots,j_e} X_1^{j_1}\cdots X_e^{j_e}.$$ After the change of variables $h_0=h_0(1,a_2,\dots,a_e)x_1^t+g$, where $g$ is a homogeneous form of degree $t$ in the ideal $(x_2,\dots, x_e)$. 
 The field $\kk$ is infinite; so, there exists a point $(a_2,\dots,a_e)$ in affine $e-1$ space with $h_0(1,a_2,\dots,a_e)\neq 0$.  \hfil \qed

\begin{chunk}\label{bi-g} If $ P=\bigoplus_i  P_i$ is a graded ring,   and $A=\bigoplus_i A_i$ and $B=\bigoplus_i B_i$ are graded $ P$-modules, then the module $\operatorname{Tor}^{ P}_{\bullet}(A,B)$ is a bi-graded $ P$-module. Indeed, if $${\mathbb Y}: \dots \to Y_1 \to Y_0\to A$$ is a  resolution of $A$ by free $P$-modules,  homogeneous of degree zero, then 
$$\operatorname{Tor}_{p,q}^ P(A,B)=\frac{\ker [(Y_p\otimes B)_q\to (Y_{p-1}\otimes B)_q]}{\operatorname{im} [(Y_{p+1}\otimes B)_q\to (Y_{p}\otimes B)_q]}.$$\end{chunk}

\begin{chunk} If $\mathbb Y$ is a complex, then we use $Z_i(\mathbb Y)$,  $B_i(\mathbb Y)$, and $\HH_i(\mathbb Y)$ to represent the modules of $i$-cycles, $i$-boundaries, and $i^{\text{th}}$-homology  of $\mathbb Y$, respectively. So, in particular, $\HH_i(\mathbb Y)=Z_i(\mathbb Y)/B_i(\mathbb Y)$.
\end{chunk}

\begin{chunk}\label{2.14} Let $(R,\m,\kk)$ be a local ring of embedding dimension $e$. The ring $R$ is {\em Golod} if 
$$P_{\kk}^R(z)= \frac{(1+z)^e}{1-\sum_{j=1}^e\dim_{\kk}\HH_j(K^R) z^{j+1}},$$ where $K^R$ is the Koszul complex on a minimal set of generators of $\m$. 
\end{chunk}

\section{Homomorphisms from a power of the maximal ideal to the socle.}\label{mult}

In order to study compressed rings, one must have an appropriate duality theory.  Partial derivatives provide the duality for Iarrobino \cite{I84}.  Fr\"oberg and Laksov \cite{FL84} and Boij and Laksov \cite{BL} pick a vector space $V$ in the polynomial ring $\kk[x_1,\dots,x_e]$ and use colon ideals to define an ideal  $I$ in the polynomial ring with the property that the corresponding quotient ring has socle $V$. The colon ideals provide the duality in these cases.  Rossi and \c Sega \cite{RS} work in a Gorenstein ring and use Gorenstein duality directly. Duality for us is supplied by homomorphisms from a power of the maximum ideal to the socle.
  
Let $(R,\m,\kk)$ be a  local Artinian ring with top socle degree $s$. If $j$ and $k$ are integers with $0\le j$, $1\le k$, and $j+k\le s+1$, then the $R$-module homomorphism
$$\mult: \m^j\cap (0:\m^k)\to \Hom_R\left(\m^{k-1},\socle(R)\cap \m^{j+k-1}\right)$$
of \ref{N27.1} induces an injective $R$-module homomorphism
\begin{equation}\label{M24.a}\frac {\m^j\cap (0:\m^k)}{\m^j\cap (0:\m^{k-1})} \to \Hom_R\left(\frac{\m^{k-1}}{\m^k},\socle(R)\cap \m^{j+k-1}\right),\end{equation}
which we also call $\mult$. The injections of (\ref{M24.a}) are our main tool for studying compressed rings.
 The remarkable feature of these injections  is that if one of them is a surjection, and all other conditions are favorable, then a whole family  of these injections are surjections. Recall the invariant ${\tt v}(R)$ from \ref{2.2}.(\ref{v(R)}).

\begin{lemma}\label{27.4}  Let $(R,\m,\kk)$ be a  local Artinian ring with embedding dimension $e$ and top socle degree $s$, and let
$A$ and $B$ be non-negative integers with $0\le A+B\le s$.
 Assume that \begin{enumerate} [\rm(a)]
\item\label{27.4.a} $B\le {\tt v}(R)-1$,  and  
\item\label{27.4.b} 
$\mult:\m^{A}\cap (0:\m^{B+1})\to \Hom_R(\m^{B},\socle(R)\cap\m^{A+B})$ is surjective.
\end{enumerate}
 Then \begin{equation}\label{goal26}\mult:\m^{A+\epsilon}\cap (0:\m^{B+1-\epsilon})\to \Hom_R(\m^{B-\epsilon},\socle(R)\cap\m^{A+B})\end{equation} is surjective for all integers $\epsilon$ with $0\le \epsilon\le \min\{B,s-A\}$.
\end{lemma}

\begin{proof} The proof may be iterated; consequently, it suffices to prove the result for $\epsilon$ equal to $1$. 
Let $x_1,\dots, x_e$ be a minimal generating set for $\m$. For each 
integer $i$, let $\binom{x_1,\dots, x_e}i$ be the set of monomials of degree $i$ in $x_1,\dots,x_e$. View $\binom{x_1,\dots, x_e}i$ as a subset of $R$. Hypothesis~(\ref{27.4.a}) guarantees that 
$\binom{x_1,\dots, x_e}i$ represents a basis for $\m^i/\m^{i+1}$
for all $i$ with $0\le i\le B$.  Let $c=\dim_{\kk} (\socle(R)\cap\m^{A+B})$ and $\sigma_1,\dots,\sigma_c$ be a basis for $\socle(R)\cap\m^{A+B}$. 
If $i\in\{B-1,B\}$, $m\in \binom{x_1,\dots, x_e}{B}$, and $\gamma$ is an integer with $1\le \gamma\le c$, then the $R$-module homomorphism $\phi_{m,\gamma}$, which is defined by 
$$\textstyle \phi_{m,\gamma}(m')=\delta_{m,m'}\cdot \sigma_{\gamma},\quad \quad \text{ for $m'\in \binom{x_1,\dots, x_e}{i}$},$$ is an element of $\Hom_R(\m^{i},\socle(R)\cap\m^{A+B})$;
furthermore,
$$\textstyle \{\phi_{m,\gamma}\mid m\in \binom{x_1,\dots, x_e}{i}\text{ and }1\le \gamma\le c\}$$ is a basis for the vector space $\Hom_R(
\m^{i},\socle(R)\cap\m^{A+B})$. In this discussion, ``$\delta$'' is  
the Kronecker delta; that is,
$$\delta_{m,m'}=\begin{cases} 1,&\text{if $m=m'$,}\\0,&\text{otherwise.}\end{cases}$$

Fix a monomial $m_0\in \binom{x_1,\dots, x_e}{B-1}$ and an index $\gamma$ with $1\le \gamma\le c$. We complete the proof by showing that  the basis element $\phi_{m_0,\gamma}$ of $\Hom_R(\m^{B-1},\socle(R)\cap\m^{A+B})$ is in the image of (\ref{goal26}) when $\epsilon=1$. Observe that $x_{1}m_0\in \binom{x_1,\dots,x_e}B$ and $\phi_{x_1m_0,\gamma}$ is in $\Hom_R(\m^{B},\socle(R)\cap\m^{A+B})$. The hypothesis guarantees that there is an element $\theta\in \m^A\cap (0:\m^{B+1})$, with $\mult_\theta= \phi_{x_1m_0,\gamma}$. Observe that 
$$x_1\theta\in \m^{A+1}\cap (0:\m^{B})\quad\text{and}\quad 
\mult_{x_1\theta}\in \Hom_R(\m^{B-1},\socle(R)\cap\m^{A+B})$$ with 
$\mult_{x_1\theta}=\phi_{m_0,\gamma}$. Indeed, 
$$\mult_{x_1\theta}(m)=mx_1\theta=\delta_{mx_1,m_0x_1}\cdot \sigma_\gamma=\delta_{m,m_0}\cdot \sigma_\gamma
=\phi_{m_0,\gamma}(m)$$ for all $m\in \binom{x_1,\dots,x_e}{B-1}$.
\end{proof}

\section{Compressed local Artinian rings.} \label{comp}

A compressed  local Artinian ring has maximal length among all  local Artinian rings with the same embedding dimension and socle polynomial. 
 Compressed algebras were introduced 
by   Iarrobino \cite{I84}. Fr\"oberg and   Laksov \cite{FL84} offer an alternate discussion, essentially from the dual point of view. 
Traditionally, the concept ``compressed'' was only defined for equicharacteristic rings. 
 However, the equicharacteristic hypothesis is irrelevant  and the proof of our main theorem (Theorem~\ref{5.1}) holds for arbitrary compressed  local Artinian rings. 

There are two themes in this section. In Theorem~\ref{BL-gts} and Remark~\ref{FL-gts} we explain the sense in which  generic standard-graded  Artinian algebras over a field are compressed. 
A short, self-contained, and  direct 
proof of Theorem~\ref{BL-gts} 
may be found in \cite
{BL}. 

In Theorem~\ref{comp-ring} and Corollary~\ref{large powers} we justify the first sentence of the present section  and we describe the annihilator of each large power of the maximal ideal of $R$ when $R$ is a compressed  local Artinian ring. 
This information is used heavily in the proofs of 
Corollary~\ref{I-cor} and  Lemma~\ref{FirstStep}. Lemma~\ref{FirstStep} is the key result from local algebra that is used in the second half of the paper about Poincar\'e series.

\begin{theorem}\cite[3.4]{BL}\label{BL-gts} Let $\kk$ be an infinite field, $(e,s,c)$ be integers with $2\le e$ and $$\textstyle 1\le c< \binom{e+s-1}s,$$  $Q$ be a standard-graded polynomial ring over $\kk$ of embedding dimension $e$, 
  $\mathcal G$ be the Grassmannian  of subspaces of $Q_s$ of codimension $c$,  and $\mathcal L$ be the set of homogeneous ideals $I$ of $Q$ 
such that $Q/I$ is a standard-graded  Artinian $\kk$-algebra with socle polynomial $cz^s$.
Then the following statements hold. \begin{enumerate}[\rm(a)]
\item The set  $\mathcal G$ parameterizes $\mathcal L$.
\item If $V$ is in $\mathcal G$, then the corresponding ideal $I$ in $\mathcal L$ is generated by 
$$
\sum_{i=1}^s(V:_{Q_i}Q_{s-i}).$$
\item If $I$ is in $\mathcal L$, then the corresponding element of $\mathcal G$ is $I_s$.
\item There is a non-empty open subset of $\mathcal G$ for which the corresponding quotient $Q/I$ is compressed.
\end{enumerate}
\end{theorem}

\begin{remark}\label{FL-gts} Let $\kk$ be an infinite field. It is shown in Section 7 of \cite{FL84}, especially Theorem 14, that generic standard-graded Artinian $\kk$-algebras are compressed for all legal socle polynomials. (Theorem~\ref{BL-gts} deals only with socle polynomials of the form $cz^s$.) The exact details of the result in \cite{FL84} are similar to, but more complicated than, the details of Theorem~\ref{BL-gts}. There is no need to record the details of the statement of \cite{FL84} in the present paper. The extra complication arises  because\begin{enumerate}[\rm(a)] \item the set of legal socle polynomials is more complicated than $$\textstyle\{cz^s\mid 1\le c<\binom{e+s-1}s\},$$ and
\item the parameterization space for the set of all Artinian $\kk$-algebras with a given embedding dimension and socle polynomial is more complicated than the 
Grassmannian of all subspaces of codimension $c$ in a vector space of dimension $\binom{e+s-1}s$. \end{enumerate}\end{remark}

\bigskip We turn to the second theme in this section which is the notion of compressed local Artinian rings which are not necessarily equicharacteristic.
Proposition~\ref{M24} is a very important step in the proof of Theorem~\ref{comp-ring}. Ultimately, we use Proposition~\ref{M24} to connect the numerical information given in the definition of compressed rings to the homomorphisms ``$\mult$'' of Section~\ref{mult}.
\begin{proposition}\label{M24} Let $(R,\m,\kk)$ be an Artinian local ring with embedding dimension $e$, top socle degree $s$, and socle polynomial $\sum_{i=0}^sc_iz^i$. If     $j$ is an arbitrary integer, with
$0\le j\le s$, 
then  $$\lambda_R (\m^j)\le 
 \sum\limits_{\ell=j}^s c_{\ell} \binom{e+\ell-j}{\ell-j}.$$
\end{proposition}
\begin{proof} Observe that \begin{align}\label{filt1}0={}&\big(\m^j\cap(0:\m^0)\big) \subseteq \big(\m^j\cap(0:\m^1)\big)\subseteq \big(\m^j\cap(0:\m^2)\big)\subseteq \cdots \\ \notag\cdots \subseteq {}&\big(\m^j\cap(0:\m^{s-j-1})\big)\subseteq \big(\m^j\cap(0:\m^{s-j})\big)\\ \notag \subseteq {}&\big(\m^j\cap(0:\m^{s-j+1})\big)=\m^j\end{align}
is a filtration of $\m^j$. The proof is obtained by exhibiting an injection from each factor of  filtration (\ref{filt1}) into a vector space whose dimension is easy to approximate.

If  $k$ is an integer with  $1\le k\le s+1-j$, then the $R$-module injection $\mult$ of (\ref{M24.a}) yields 
$$\lambda_R\left(\frac {\m^j\cap (0:\m^k)}{\m^j\cap (0:\m^{k-1})}\right) \le 
\left(\dim_{\kk}\left(\frac{\m^{k-1}}{\m^k}\right)\right) \left(\dim_{\kk}(\socle(R)\cap 
\m^{j+k-1})\right).$$ Recall that  
\begin{equation}\label{Recall}\dim_{\kk}\left(\frac{\m^{k-1}}{\m^k}\right)\le \binom{(e-1)+(k-1)}{k-1},\end{equation} because $\m^{k-1}$ is  generated by the set of monomials of degree $k-1$ in any minimal generating set of $\m$, and  
$$\dim_{\kk}(\socle(R)\cap \m^{j+k-1})=\sum\limits_{\ell=j+k-1}^s c_{\ell},$$by the definition of socle polynomial.
Thus, 
\begin{align}\label{M24.b}\lambda_R\left(\frac {\m^j\cap (0:\m^k)}{\m^j\cap (0:\m^{k-1})}\right)
\le \sum_{\ell=j+k-1}^s c_{\ell} \binom{(e-1)+(k-1)}{k-1},
 \end{align}
for $0\le j\le s$ and $1\le k\le s-j+1$. Combine (\ref{filt1}) and (\ref{M24.b}) to obtain $$\lambda_R (\m^j)\le 
\sum\limits_{k=1}^{s-j+1} \sum\limits_{\ell=j+k-1}^s c_{\ell} \binom{(e-1)+(k-1)}{k-1}.$$ Let $K=k+j-1$; reverse the order of summation; let $\alpha=K-j$; and recall the relationship between the number of monomials of degree at most  $\ell-j$ in $e$ variables and the number of monomials of degree equal to $\ell-j$ in $e+1$ variables to conclude
\begin{align}\label{M24.c}\lambda_R (\m^j)&{}\le 
\sum\limits_{K=j}^{s} \sum\limits_{\ell=K}^s c_{\ell} \binom{(e-1)+(K-j)}{K-j}\\\notag&{}
=\sum\limits_{\ell=j}^sc_{\ell}\sum\limits_{K=j}^{\ell}   \binom{(e-1)+(K-j)}{K-j}
\\\notag&{}=\sum\limits_{\ell=j}^sc_{\ell}\sum\limits_{\alpha=0}^{\ell-j}   \binom{(e-1)+\alpha}{\alpha}=\sum\limits_{\ell=j}^sc_{\ell} \binom{e+\ell-j}{\ell-j} .\end{align}
\vskip-18pt\end{proof}

In Theorem~\ref{comp-ring}  we justify the claim that
a compressed  local Artinian ring has maximal length among all  local Artinian rings with the same embedding dimension and socle polynomial. 
The proof of Theorem~\ref{comp-ring} contains a wealth of information. We mine this information throughout the rest of the section.

\begin{theorem}\label{comp-ring}
Let  
$(R,\m,\kk)$ be a local Artinian ring with embedding dimension $e$, top socle degree $s$, and socle polynomial  $\sum_{i=0}^sc_iz^i$. 
Then the following statements hold.
\begin{enumerate}[\rm(a)]
\item\label{3.5.a} The length of $R$ satisfies
\begin{equation}\label{,max-len}\textstyle
\lambda_R(R)\le \sum\limits_{i=0}^s \min \Big\{ \binom{(e-1)+i}{i}, \ 
\sum\limits_{\ell=i}^s c_\ell\binom{(e-1)+(\ell-i)}{\ell-i}
\Big\}.\end{equation}

\item\label{3.5.b}  Equality holds  in {\rm(\ref{,max-len})} if and only if  $R$ is a compressed  local Artinian ring in the sense of Definition~{\rm\ref{cmped-new*}}.
\item \label{xxx.a} 
If $R$ is a compressed  local Artinian ring, then
 the parameter ${\tt v}(R)$ from {\rm\ref{2.2}.(\ref{v(R)})} satisfies
$s\le 2{\tt v}(R)-1$. 
\end{enumerate}
\end{theorem}

\begin{remark-no-advance}\label{4.4.2}Equation~(\ref{,light}),  the proof of assertion (\ref{xxx.a}), and the identity 
$$\textstyle\sum\limits_{i=0}^{{\tt v}(R)-1}\binom{(e-1)+i}{i}= \binom{e+{\tt v}(R)-1}{{\tt v}(R)-1}$$ show that an alternate version of (\ref{,max-len}) is given by
\begin{equation}\label{notag}\lambda_R(R)\le \binom{e+{\tt v}(R)-1}{{\tt v}(R)-1}+ \sum\limits_{\ell={\tt v}(R)}^sc_\ell \binom{e+\ell-{\tt v}(R)}{\ell-{\tt v}(R)}. 
\end{equation}Once the proof of Theorem~\ref{comp-ring} is complete, then we know that a local Artinian ring $R$ is compressed if and only if equality holds in (\ref{notag}). This observation  provides an effective method for testing if a ring is compressed.
\end{remark-no-advance}

\begin{proof} 
Define  $t$ to be  the integer  
\begin{equation}\label{3.5.0}\textstyle t=\min\Big\{i\Big \vert \sum\limits_{\ell=i}^s c_\ell\binom{(e-1)+(\ell-i)}{\ell-i}< \binom{(e-1)+i}{i}\Big\}.\end{equation} Observe that
\begin{align}\notag\textstyle
\binom{(e-1)+i}i&\textstyle\le \sum\limits_{\ell=i}^s c_\ell\binom{(e-1)+(\ell-i)}{\ell-i},&&\text{for $0\le i\le t-1$, and}\\\textstyle
\label{second}\sum\limits_{\ell=i}^s c_\ell\binom{(e-1)+(\ell-i)}{\ell-i}&\textstyle<\binom{(e-1)+i}i,&&\text{for $t\le i\le s$};
\end{align}
hence, the inequality (\ref{,max-len}) may be re-written as 
 \begin{equation}\notag\textstyle \lambda_R(R)\le \sum\limits_{i=0}^{t-1}\binom{(e-1)+i}{i}+ \sum\limits_{i=t}^{s}\sum\limits_{\ell=i}^s c_\ell\binom{(e-1)+(\ell-i)}{\ell-i}
.\end{equation}
Reverse the order of summation, let $\alpha=\ell-i$, and count the number of monomials of degree at most $\ell-t$ in $e$ variables to see that
\begin{align*}&\sum\limits_{i=t}^{s}\sum\limits_{\ell=i}^s c_\ell\binom{(e-1)+(\ell-i)}{\ell-i}
=\sum\limits_{\ell=t}^sc_\ell \sum\limits_{i=t}^{\ell} \binom{(e-1)+(\ell-i)}{\ell-i}\\
={}&\sum\limits_{\ell=t}^sc_\ell \sum\limits_{\alpha=0}^{\ell-t} \binom{(e-1)+\alpha}{\alpha}
=\sum\limits_{\ell=t}^sc_\ell \binom{e+\ell-t}{\ell-t};\end{align*}
and therefore, the inequality (\ref{,max-len}) is equivalent to
\begin{equation}\label{,light}\textstyle \lambda_R(R)\le \sum\limits_{i=0}^{t-1}\binom{(e-1)+i}{i}+ \sum\limits_{\ell=t}^sc_\ell \binom{e+\ell-t}{\ell-t}
.\end{equation}
On the other hand, the inequality (\ref{,light}) does indeed hold, because 
$$\lambda_R(R)=\sum\limits_{i=0}^{t-1}\lambda_R(\m^i/\m^{i+1})+\lambda_R(\m^t);$$ 
 \begin{equation}\label{M26.b}\lambda_R(\m^i/\m^{i+1})\le \binom{(e-1)+i}{i},\end{equation} as described at (\ref{Recall});  and 
Proposition~\ref{M24} guarantees that \begin{equation}\label{M26.c}\lambda_R (\m^t)\le 
 \sum\limits_{\ell=t}^s c_{\ell} \binom{e+\ell-t}{\ell-t}.\end{equation} This completes the proof of (\ref{3.5.a}).

\bigskip The parameter $c_s$ is at least $1$; so, one consequence of the inequality (\ref{second}), when $i=t$, is 
$$\textstyle \binom{(e-1)+s-t}{s-t}\le 
c_s\binom{(e-1)+s-t}{s-t}\le\sum\limits_{\ell=t}^s c_\ell\binom{(e-1)+(\ell-t)}{\ell-t}<\binom{(e-1)+t}t.$$The binomial coefficient $\binom{e-1+i}{i}$ counts the number of monomials of degree $i$ in a polynomial ring with $e$ variables; therefore, the most recent inequality forces $s-t<t$; and therefore, \begin{equation}\label{s<=2t-1}s\le 2t-1.\end{equation}

\bigskip 
\noindent (\ref{3.5.b}) 
It is clear that if  $R$ is a compressed  local Artinian ring in the sense of Definition~{\rm\ref{cmped-new*}}, then  equality holds  in (\ref{,max-len}).  

Henceforth, in this proof, we assume that equality holds in (\ref{,max-len}).
We first prove that $R$ is a compressed  local Artinian ring in the sense of  Definition~{\rm\ref{cmped-new*}}; that is, we prove that
\begin{equation}\dim_{\kk}(\m^i/\m^{i+1})=\begin{cases} 
\binom{e-1+i}i,&\text{if $0\le i\le t-1$, and}\\
\sum\limits_{\ell=i}^s c_\ell\binom{(e-1)+(\ell-i)}{\ell-i}, &\text{if $t\le i\le s$}.
\end{cases}\label{26goal}\end{equation}

The inequality (\ref{,light})
is equivalent to (\ref{,max-len}); hence      equality holds in (\ref{,light}) and in all of the intermediary inequalities that lead to  (\ref{,light}). 
In particular, 
\begin{equation}\label{LIGHT}\textstyle \lambda_R(R)= \sum\limits_{i=0}^{t-1}\binom{(e-1)+i}{i}+ \sum\limits_{\ell=t}^sc_\ell \binom{e+\ell-t}{\ell-t}
\text{ and}\end{equation}
\begin{align}\label{early}\textstyle \dim_{\kk}(\m^i/\m^{i+1})=\binom{e-1+i}i,\quad\text{for $0\le i\le t-1$,}\end{align}follow from (\ref{M26.b}), 
and
\begin{equation}\label{M26.d}\lambda_R (\m^t)= 
 \sum\limits_{\ell=t}^s c_{\ell} \binom{e+\ell-t}{\ell-t}\end{equation}
follows from (\ref{M26.c}). The equality (\ref{M26.d}) forces equality to hold in (\ref{M24.c}) when ${j=t}$; hence equality holds in (\ref{M24.b}) when ${j=t}$; and therefore, the injections of (\ref{M24.a}) are isomorphisms when $j=t$ and $1\le k\le s+1-t$.

We apply Lemma~\ref{27.4} to each pair $(A,B)$ with $A=t$ and $0\le B\le s-t$. 
Recall from (\ref{s<=2t-1}) that $s-t\le t-1$; hence (\ref{early}) ensures that hypothesis~(\ref{27.4.a}) of Lemma~\ref{27.4} is in effect. The isomorphisms of 
(\ref{M24.a}), for $j=t$, ensure that hypothesis~(\ref{27.4.b}) is in effect. Conclude that 
\begin{equation}\label{2.5.11.5}\mult: \m^j\cap (0:\m^k)\to \Hom_{R}(\m^{k-1},\socle(R)\cap \m^{j+k-1})\end{equation} is surjective for all $j,k$ with 
\begin{equation}\label{jk}t\le j\le s\quad\text{and}\quad 1\le k\le s-j+1.\end{equation}
Therefore, the injections of (\ref{M24.a}) are isomorphisms and equality holds in (\ref{M24.b}) when $j$ and $k$ satisfy (\ref{jk}); that is, 
\begin{align}\label{M24.a-iso}&\mult: \frac {\m^j\cap (0:\m^k)}{\m^j\cap (0:\m^{k-1})} \xrightarrow{\ \cong\ } \Hom_R\left(\frac{\m^{k-1}}{\m^k},\socle(R)\cap \m^{j+k-1}\right)\\\notag&\text{for $ 
t\le j\le s$\quad\text{and}\quad $1\le k\le s-j+1$.}
\end{align}
Furthermore, equality holds in (\ref{M24.c}) for $t\le j\le s$. In particular,
\begin{align}\notag\dim_{\kk}\m^j/\m^{j+1}={}&\sum\limits_{\ell=j}^sc_{\ell} \binom{e+\ell-j}{\ell-j}-
\sum\limits_{\ell=j+1}^sc_{\ell} \binom{e+\ell-j-1}{\ell-j-1}\\={}&
\sum\limits_{\ell=j}^sc_{\ell} \binom{e+\ell-j-1}{\ell-j},\label{2.5.13}\end{align} for $t\le j\le s$.
Combine (\ref{early}) and (\ref{2.5.13}) to see that (\ref{26goal}) holds.
This completes the proof of (\ref{3.5.b}).

\bigskip \noindent(\ref{xxx.a}) The inequalities of (\ref{second}) and (\ref{s<=2t-1}) hold because of the definition of $t$ which is given in (\ref{3.5.0}). We assume equality holds in (\ref{,max-len}); so (\ref{26goal}) holds. We conclude that $t={\tt v}(R)$ and $s\le 2{\tt v}(R)-1$.
\end{proof}

In  Corollary~\ref{large powers} we 
 describe the annihilator of each large power of the maximal ideal of $R$, when $R$ is a compressed  local Artinian ring.  
This information is used heavily in the proofs of  Corollary~\ref{I-cor} and Lemma~\ref{FirstStep}.

\begin{corollary}\label{large powers}
If  $(R,\m,\kk)$ is a compressed  local Artinian ring with  top socle degree $s$, then the following statements hold.
\begin{enumerate}[\rm(a)]
\item\label{2.5.cii} If ${\tt v}(R)\le j\le s$, then $(0:\m^j)=\m^{s-j+1}$.
\item\label{2.8} If $1\le j\le s+1$, then 
$\m^{j}:\m=\m^{j-1}+\socle(R)$.
\item\label{4.5.c} If $R$ is also a level ring {\rm(}that is, if $\socle(R)=\m^s${\rm)}, then $$(0:\m^j)=\m^{s-j+1}\quad\text{ for $0\le j\le s+1$.}$$
\end{enumerate}
\end{corollary}

\begin{proof} The ring $R$ is compressed, local, and Artinian; consequently, $R$ satisfies all of the statements in the proof of Theorem~\ref{comp-ring} and, according to  Theorem~\ref{comp-ring}.(\ref{xxx.a}),  the parameter ${\tt v}(R)$ is equal to the ``$t$'' of (\ref{3.5.0}).

\medskip\noindent 
(\ref{2.5.cii})   Apply (\ref{2.5.11.5}) with $t\le j\le s$ and $k=s-j+1$ to see that
$$\mult:\m^j \to \Hom_R\left(\frac{\m^{s-j}}{\m^{s-j+1}},\m^s\right)$$ is a surjection. It follows that 
$$\textstyle \frac{\m^{s-j}\cap (0:\m^j)}{\m^{s-j+1}}\subseteq \bigcap\ker f=0,$$
as $f$ roams over $\Hom_R\left(\frac{\m^{s-j}}{\m^{s-j+1}},\m^s\right)$. The final equality is a statement about homomorphisms of  vector spaces. 
Thus, 
$$\m^{s-j}\cap (0:\m^j)\subseteq \m^{s-j+1};$$hence,
$$\m^{s-j+1}\subseteq \m^{s-j}\cap (0:\m^j)\subseteq \m^{s-j+1}\quad \text{and}$$
\begin{equation}\label{so far} t\le j\le s\implies \m^{s-j}\cap (0:\m^j)=\m^{s-j+1}.\end{equation} Apply descending induction on $j$ to see that
\begin{equation}\label{goalm30} t\le j\le s\implies (0:\m^j)=\m^{s-j+1}.\end{equation}
Indeed, (\ref{goalm30}) holds when $j=s$. Assume that (\ref{goalm30}) holds when $j+1$. We prove that (\ref{goalm30}) holds when $j$. Observe that 
\begin{equation}\label{eol}(0:\m^j)\subseteq (0:\m^{j+1})= \m^{s-j}.\end{equation}(The final equality is due to the induction hypothesis.) Thus,
$$(0:\m^j)=(0:\m^j)\cap \m^{s-j}=\m^{s-j+1}.$$ The equality on the left is due to (\ref{eol}) and the equality on the right is due to (\ref{so far}). 

\bigskip \noindent(\ref{2.8}) 
We saw in Observation~\ref{obs} that 
$$(\m^j:\m)=\m^{j-1}=\m^{j-1}+\socle(R),$$ for $1\le j\le {\tt v}(R)$. Also, the assertion of (\ref{2.8}) is obvious at $j=s+1$.
The parameter ${\tt v}(R)$ continues to equal to the ``$t$'' of (\ref{3.5.0}). 
We prove that 
if $t+1\le j\le s$, then 
$$\m^{j}:\m=\m^{j-1}+\socle(R).$$
It suffices to prove the inclusion ``$\subseteq$''. To do this, it suffices to prove the following claim.

\medskip
\noindent Claim. {\it If $2\le a\le s-j+2$ and $\theta\in (\m^{j}:\m) \cap(0:\m^a)$, then there exists an element $\theta'$ in $\m^{j-1}\cap (0:\m^a)$ with $\theta-\theta'\in (0:\m^{a-1})$.}

\medskip\noindent We prove the claim. Observe that multiplication by $\theta$ is an element of 
$$\Hom_R(\m^{a-1}/\m^a,\socle(R)\cap \m^{a+j-2}).$$ Of course, we know from 
(\ref{M24.a-iso}),  that there is an element $\theta'\in \m^{j-1}\cap (0:\m^a)$ with 
multiplication by $\theta'$  equal to multiplication by $\theta$ on $\m^{a-1}$.

\bigskip \noindent(\ref{4.5.c}) One direction of assertion (\ref{4.5.c}) is obvious. We prove the other direction. The special hypothesis $\socle(R)=\m^s$ of (\ref{4.5.c}) guarantees that $\socle (R)\subseteq \m^a$ for $0\le a\le s$; and therefore, under this special hypothesis, assertion (\ref{2.8}) becomes \begin{equation}\label{euq} 1\le j\le s+1 \implies 
\m^{j}:\m=\m^{j-1}.\end{equation} Fix an element $x$ in $R$ and an integer $i$ with $0\le i\le s$ and $x\m^i=0$. We use descending induction to prove  that 
\begin{equation}\label{June24} 0\le a\le i\implies x\m^a\subseteq \m^{s+a+1-i}.\end{equation} It is clear that (\ref{June24}) holds at $a=i$. Suppose $1\le a\le i$ and (\ref{June24}) holds at $a$. Then $$x\m^{a-1}\subseteq (\m^{s+a+1-i}:\m)= \m^{s+a-i}.$$ (Use the induction hypothesis (\ref{June24}) for the inclusion and   (\ref{euq}) for the equality.) Thus (\ref{June24}) holds at $a=i-1$ and the induction is complete. Apply (\ref{June24}) at $a=0$ to see that $x\in \m^{s+1-i}$.
\end{proof}

Associated graded objects are discussed in \ref{agr}. The following result is shown by
Iarrobino \cite[Cor.~3.8]{I84} 
in the equicharacteristic case. This result is confirmation that
Definition~\ref{cmped-new*} (or equivalently, equality in (\ref{,max-len}), or equivalently, equality in  (\ref{notag}))
is  the correct intrinsic definition of a compressed local Artinian ring.

\begin{corollary}\label{I-cor} 
Let $(R,\m,\kk)$ be a 
local Artinian ring.
Then 
$R$ is a  
compressed    
ring  if and only if the associated graded ring $\agr{R}$ is a compressed ring 
and $R$ and $\agr{R}$ have the same socle polynomial. 
\end{corollary}

\begin{proof} $(\Leftarrow)$ This direction is obvious. Indeed, the Hilbert function of $R$ is always equal to the Hilbert function of $\agr{R}$ and the hypothesis asserts that the relationship of Definition~\ref{cmped-new*}
holds between $h_{\agr{R}}$ and the socle polynomial of $\agr{R}$.

\smallskip\noindent$(\Rightarrow)$  As described above, it suffices to show $R$ and $\agr{R}$ have the same socle polynomial. 
The isomorphism theorem $I/(I\cap J)\cong (I+J)/J$ ensures that 
$$\frac{\socle(R)\cap \m^i}{\socle(R)\cap \m^{i+1}}\cong \frac{(\socle(R)\cap \m^i)+\m^{i+1}}{\m^{i+1}};$$hence the socle polynomial $R$, defined in \ref{2.2}.(\ref{2.2.d}), is also  equal to 
$$\sum\limits_{i=0}^s \dim_{\kk}\frac{(\socle(R)\cap \m^i)+\m^{i+1}}{\m^{i+1}} z^i,$$where $s$ is the top socle degree of $R$.
On the other hand, the socle polynomial of the graded local ring $\agr{R}$ is 
$$\sum\limits_{i=0}^s \dim_{\kk}\frac{\m^i\cap (\m^{i+2}: \m)}{\m^{i+1}} z^i.$$ The ring $R$ is compressed; hence Corollary~\ref{large powers}.(\ref{2.8}) guarantees that  
\begin{equation}\notag\m^{j}:\m=\m^{j-1}+\socle(R)\quad\text{for $1\le j\le s$}\end{equation} and the two socle polynomials are equal.
\end{proof}

The statement of the main result, Theorem~\ref{5.1},  depends on the relationship between $s$ and the invariant ${\tt v}(R)$. 
 Recall from Theorem~\ref{comp-ring}.(\ref{xxx.a}) that ${s\le 2{\tt v}(R)-1}$. We show in Observation~\ref{xxx.b} that  the critical situation is 
${s= 2{\tt v}(R)-1}$. 
The first step in the critical situation is taken in Lemma~\ref{FirstStep}. This step is the main ingredient  in the proof of Lemma~\ref{4.2.c}.

\begin{lemma}\label{FirstStep}Let
$(R, \m,\kk)$ be a compressed local Artinian ring
with 
embedding dimension $e$ and 
top socle degree $s$.   
Assume  that  
$s$ is odd and that $s=2{\tt v}(R)-1$. 
Decompose the maximal ideal $\m$ as the sum of two subideals  $\m=(x_1)+\m'$ with $x_1$ a minimal generator of $\m$ and  $\mu(\m')=e-1$. Then
$$
x_1^{\frac{s-1}2} [\ann_R(\m') \cap \m^{\frac{s+1}2}] = \m^s.
$$\end{lemma}
\begin{proof}Let $t$ denote ${\tt v}(R)$, which by hypothesis is equal to $(s+1)/2$. It is clear that
$$x_1^{t-1} [\ann_R(\m') \cap \m^t] \subseteq \m^s.$$
For the other direction, let $\sigma$ be an element of $\m^s$.
We will construct an element $\Theta$ of $\ann_R(\m') \cap \m^t$ such that $x_1^{t-1}\Theta=\sigma$. We build $\Theta$ as $\theta_0+\cdots+\theta_{t-2}$, where, for each $i$, 
\begin{equation}\label{build}\begin{cases}
\theta_i\in \m^t\cap (0:\m^{t-i}),\\
(\theta_0+\cdots+\theta_i)x_1^{t-1}=\sigma,&\text{and}\\
(\theta_0+\cdots+\theta_i)\m'\m^{t-i-2}=0.\end{cases}\end{equation}

We first build $\theta_0$. 
 Consider the homomorphism $$\phi_0\in \Hom_R(\m^{t-1}/\m^t, \socle(R)\cap \m^{s}),$$ which is given by $$\phi_0(\overline{\m'\m^{t-2}})=0\quad\text{and}\quad \phi_0(\overline{x_1^{t-1}})=\sigma.$$ (Keep in mind that $\m^{t-1}/\m^t$ and $\overline{\m'\m^{t-2}}\oplus \kk\overline{x_1^{t-1}}$
are isomorphic 
as $R$-modules. At this point $\bar{\phantom{x}}$ means mod $\m^t$.) Apply (\ref{M24.a-iso}), with $j=k=t$ to obtain an element $\theta_0\in \m^t\cap \ann (\m^t)$ with $x_1^{t-1}\theta_0=\sigma$ and $\theta_0\m'\m^{t-2}=0$.

Suppose $0\le i\le t-3$ and elements $\theta_0,\dots \theta_i$, which satisfy (\ref{build}), have been identified. We now build $\theta_{i+1}$. Consider the homomorphism $$\phi_{i+1}\in \Hom_R(\m^{t-i-2}/\m^{t-i-1}, \socle(R)\cap \m^{s-i-1}),$$
which is given by $$\phi_{i+1}(\bar u)=-(\theta_0+\cdots+\theta_i)u,\quad\text{for $u\in 
\m'\m^{t-i-3}$,\quad and}\quad \phi_{i+1}(\overline{x_1^{t-i-2}})=0.$$ (At this point $\bar{\phantom{x}}$ means mod $\m^{t-i-1}$. We have taken advantage of a direct sum decomposition of $\m^{t-i-2}/\m^{t-i-1}$ to define $\phi_{i+1}$. The image of $\phi_{i+1}$ is contained in the socle of $R$ because of the properties of the earlier $\theta$'s as described in (\ref{build}).) Apply (\ref{M24.a-iso}), with $j=t$ and $k=t-i-1$ to obtain an element 
$$\begin{cases} 
\theta_{i+1}\in \m^t\cap \ann \m^{t-i-1}&\text{with}\\
(\theta_0+\cdots +\theta_{i+1})x_1^{t-1}=\sigma&\text{and}\\
(\theta_0+\cdots +\theta_{i+1})(\m'\m^{t-i-3})=0.
\end{cases}$$Iterate this procedure to find $\theta_{t-2}$ and thereby complete the proof.
\end{proof}

\section{Golod homomorphisms.}\label{GH}

In this paper 
 we exhibit a Golod homomorphism from a complete intersection onto  a compressed local Artinian ring $R$ and then use facts about Golod homomorphisms to draw conclusions about the Poincar\'e series of $R$-modules. 
The present section is mainly concerned with 
techniques from homological algebra that can be used to prove that a homomorphism is Golod.
The hypothesis ``compressed'' is not used anywhere in the present section.

There are  numerous 
 definitions of Golod 
homomorphism (see for example \cite
{Av86}); we give the version involving trivial Massey operations,  
found, for example, in \cite{G72}. 
In Lemma~\ref{RS1.2} we record a result from \cite{RS} which  shows how to
use homological algebra to prove  that trivial Massey operations exist.
Most of the section is about homological algebra. Indeed, in Lemmas \ref{Liana!} and \ref{4.2.a} we prove that various maps of $\Tor$ are zero. 
Lemma~\ref{4.2.a} is used in Observation~\ref{xxx.b} to show that  
if the top socle degree of a local Artinian ring $R$ is small compared to the invariant ${\tt v}(R)$ of \ref{2.2}.(\ref{v(R)}), then $R$ is a Golod ring. 
 Lemmas~\ref{6.19} and \ref{6.20} are a short study of the effect on $\Tor$ associated to taking a hypersurface section.
 The section concludes with Theorem~\ref{Levin} which is a well-known result that 
exhibits the common denominator for all Poincar\'e series $P^R_M(z)$ when there is a Golod homomorphism from a local hypersurface ring onto $R$ and $M$ roams over all finitely generated $R$-modules.
 \begin{definition}\cite[1.1]{RS} Let  $\kappa:(P,\mathfrak p,\kk)\to (R,\mathfrak m,\kk)$ be a 
surjective homomorphism of local rings and  $\mathcal D$ be an associative,  graded-commutative, Differential Graded (DG) algebra with divided powers,  which is also a homogeneous minimal resolution of $\kk$ by free $P$-modules.
Let $(\mathcal A,\partial)$ denote $\mathcal D\t_PR$. If $x$ is a homogeneous element in $\mathcal A_\ell$, then let $|x|$ denote the degree $\ell$ of $x$ and  $\bar x$ denote $(-1)^{|x|+1}x$.
Let ${\mathbf {h}}=\{h_i\}_{i\ge 1}$ be a homogeneous basis of the graded $\kk$-vector space $\HH_{\ge1}(\mathcal A)$. 
The homomorphism $\kappa:P\to R$ is {\it Golod} 
if there is a function $\mu: \bigsqcup_{n=1}^\infty {\mathbf {h}}^n \to\mathcal A$
which satisfies: \begin{enumerate}[\rm(a)]
\item $\mu(h)$  is a cycle in the homology class of $h$ for each $h\in  {\mathbf {h}}$,
\item $\partial \mu(h_1,\dots,h_n)= \sum_{i=1}^{n-1}\overline{\mu(h_1,\dots, h_i)}\mu(h_{i+1}, \dots , h_n)$ for each $n$ with $2\le n$, and
\item $\mu({\mathbf {h}}^n)\subseteq \mathfrak m \mathcal A$ for each positive $n$. 
\end{enumerate}\end{definition}

Lemma~\ref{RS1.2} is our main tool for proving that a homomorphism is Golod.

\begin{lemma}\cite[Lem.~1.2]{RS}\label{RS1.2} Let  
$\kappa:(P,\mathfrak p,\kk)\to (R,\mathfrak m,\kk)$ be 
 a  
surjective homomorphism of local rings.  
If there exists a positive integer $a$ such that{\rm:}
\begin{enumerate}[\rm(a)]
\item\label{C4.2.a} the map $\Tor^ P_i(R,\kk)\to\Tor^ P_i(R/\mathfrak m^a,\kk)$, induced by the canonical quotient map $R\to R/\mathfrak m^a$, is zero for all positive $i$, and
\item\label{C4.2.b} the map $\Tor^P_i(\mathfrak m^{2a},\kk)\to \Tor^P_i(\mathfrak m^a,\kk)$, induced by the inclusion $\m^{2a}\subseteq \m^a$, 
is zero for all non-negative integers $i$,\end{enumerate}
then $\kappa$ is a Golod homomorphism.
\end{lemma}

It is convenient to name the following  family of maps of $\Tor$.
\begin{definition}\cite[1.3.1]{RS} \label{y(1.3.1)} If $M$ is a module over the local ring $(R,\mathfrak m,\kk)$, then let
$\nu^R_i(M)$ represent the $R$-module homomorphism
$$\nu^R_i(M):\Tor_i^R(\mathfrak m M,\kk)\to \Tor_i^R(M,\kk),$$which is induced by the inclusion $\mathfrak mM\subseteq M$.
\end{definition}

We 
use  Lemma~\ref{Liana!}  to calculate $\nu_i$.
Associated graded objects are discussed in \ref{agr}.

\begin{lemma}\label{Liana!}
Let $(Q,\n,\kk)$ be a regular local ring,   $(R,\m,\kk)$ be the local ring $R=Q/I$ for some ideal $I$ of $Q$, and  $i$ and $\ell$  be two integers. 
If $\Tor_{i,j}^{\agr Q}(\agr R,\kk)=0$ for all $j$ with $\ell +1+i\le j$, then the map $$\nu_i^Q(\m^\ell)\colon \Tor_i^Q(\m^{\ell+1},\kk)\to \Tor_i^Q(\m^{\ell},\kk)$$ is identically zero. 
\end{lemma}

\begin{proof}
Let $K^R$ denote the Koszul complex over $R$ on a minimal generating set $x_1, \dots,\allowbreak x_e$ of $\m$. We  identify $\nu_i^Q(\m^\ell)$ with the map $\HH_i(\m^{\ell+1}K^R)\to \HH_i(\m^\ell K^R)$ induced by the inclusion $$\m^{\ell+1}K^R\subseteq \m^\ell K^R.$$ Let $Z$ denote the module of cycles in degree $i$ of $\m^{\ell+1}K^R$ and  $B$ denote the module of boundaries of degree $i$ in $\m^{\ell}K^R$. Note that $B\subseteq Z$. To show that $\nu_i^Q(\m^\ell)$ is zero, we need to show that $Z\subseteq B$. 
We will show that $Z\subseteq B+\m^jK^R_i$ for all $j$ with $\ell+2\le j$.

For each $j$, let $x_j^*$ denote the image the element $x_j$ in $\m/\m^2=(\agr R)_1$. Let $L$ denote the graded Koszul complex over $\agr R$ on $x_1^*, \dots, x_e^*$. When writing  $L_{p,q}$, the  index $p$ stands for the homological degree and the index $q$ for the internal degree. 
Note that  $L$ can be thought of as the associated graded complex of $K^R$, with respect to the standard $\m$-adic filtration of $K^R$.  In particular, $L_p=(\agr{(K^R_p)})(-p)$ for each $p$,  and the differential $d_L$ of $L$ is induced from the differential $d_{K^R}$ of $K^R$ as follows: If $y\in \m^qK^R_p$ and $y^*$ is the image of $y$ in $\m^qK^R_p/\m^{q+1}K^R_{p}=L_{p,p+q}$, then $d_L(y^*)$ is equal to the image of $d_{K^R}(y)$ in $\m^{q+1}K^R_{p-1}/\m^{q+2}K^R_{p-1}=L_{p-1, p+q}$. 
We  identify $\Tor^{\agr Q}(\agr R,\kk)$ with the homology of the complex $L$.

Fix an integer $p$ with $\ell+1\le p$ and let $z\in Z\cap \m^{p}K^R_i$. In particular,  $d_{K^R}(z)=0$.  We consider $z^*$ to be the image of $z$ in $\m^{p}K^R_i/\m^{p+1}K^R_{i}=L_{i,p+i}$ and note that $d_L(z^*)=0$ because $d_{K^R}(z)=0$. The hypothesis that $\Tor_{i, p+i}^{\agr Q}(\agr R,\kk)=0$ implies that $z^*=d_L(y^*)$ where $y^*\in  \m^{p-1}K^R_{i+1}/\m^{p}K^R_{i+1}=L_{i+1,p+i}$ is the image of an element $y\in \m^{p-1} K^R_{i+1}$. It follows that $z-d_{K^R}(y)\in \m^{p+1}K^R_i$, and we conclude that $z\in  B+\m^{p+1}K^R_i$. It follows that 
\begin{equation}
\label{inclusion-2}
Z\cap \m^{p}K^R_i\subseteq Z\cap (B+\m^{p+1}K^R_i)=B+Z\cap \m^{p+1}K^R_i.
\end{equation}
 Since $Z=Z\cap \m^{\ell+1}K^R_i$, we conclude inductively, using (\ref{inclusion-2}), that  $Z\subseteq B+\m^jK^R_i$ for all $j$ with $\ell+2\le j$, hence $Z\subseteq B$ by the Krull intersection Theorem. 
\end{proof}

Lemma~\ref{4.2.a} is a straightforward consequence of Lemma~\ref{Liana!}. This result is used in the proof of Observation~\ref{xxx.b}; furthermore, a consequence of Lemma~\ref{4.2.a} is restated as \ref{repeat}.

\begin{lemma}\label{4.2.a}
Let $(Q,\n,\kk)$ be a regular local ring and   $(R,\m,\kk)$ be the local ring $R=Q/I$ for some ideal $I$ of $Q$. 
Then the maps
\begin{equation}\label{zero}\Tor_i^Q(R/\mathfrak m^{\ell+1},\kk)\to \Tor_i^Q(R/\mathfrak m^{\ell},\kk)\end{equation} 
and 
 \begin{equation}\label{3.5.1}\Tor_i^Q(R,\kk) \to \Tor_i^Q(R/\m^\ell,\kk),\end{equation}
 are each  the zero map for all $(i,\ell)$ with $1\le i$ and $1\le \ell\le {\tt v}(R)-1$. The map of {\rm(\ref{zero})} is induced by the natural quotient map $R/\mathfrak m^{\ell+1}\to R/\mathfrak m^{\ell}$ and the map of {\rm(\ref{3.5.1})} is induced by the natural quotient map $R\to R/\m^\ell$.
\end{lemma}
\begin{proof}
 It is clear that 
$\Tor_{i,j}^{\agr{Q}}(\agr{Q},\kk)=0$ for all $(i,j)\neq (0,0)$. 
The parameter $\ell$ is non-negative; so Lemma~\ref{Liana!} 
yields that 
$$\nu_i^Q(\mathfrak n^{\ell}):\Tor_i^Q(\mathfrak n^{\ell+1},\kk)\to \Tor_i^Q(\mathfrak n^{\ell},\kk)$$ is the zero map for all non-negative $i$. The long exact sequences of $\Tor$ which correspond to the commutative diagram$$\xymatrix{
0\ar[r]&\mathfrak n^{\ell+1}\ar[r]\ar[d]&Q\ar[r]\ar[d]&Q/\mathfrak n^{\ell+1}\ar[r]\ar[d]&0\\
0\ar[r]&\mathfrak n^{\ell}\ar[r]&Q\ar[r]&Q/\mathfrak n^{\ell}\ar[r]&0}$$ yields that 
$$\Tor_i^Q(Q/\mathfrak n^{\ell+1},\kk)\to \Tor_i^Q(Q/\mathfrak n^{\ell},\kk)$$ is the zero map for all positive $i$. The ideal $I$ is contained in $\mathfrak n^{\ell+1}$ and $\mathfrak n^{\ell}$; so $Q/\mathfrak n^{\ell+1}=R/\mathfrak m^{\ell+1}$ and $Q/\mathfrak n^{\ell}=R/\mathfrak m^{\ell}$. Thus, (\ref{zero}) is the zero map for all positive $i$. The map of (\ref{3.5.1}) factors through (\ref{zero}).
\end{proof}

Observation~\ref{xxx.b} takes care of the ``easy case'' in the proof of the main theorem, which is Theorem~\ref{5.1}.
\begin{observation}\label{xxx.b}
Let $(R,\m, \kk)$ be a  
 local Artinian ring
with  
top socle degree $s$.
 If $s\le 2{\tt v}(R)-3$, then $R$ is a Golod ring.
\end{observation}

\begin{proof} Let  
$t$ denote ${\tt v}(R)$. 
The ring $R$ is complete and local; so the Cohen structure theorem guarantees that there is a regular local ring $(Q,\n,\kk)$ with   $R=Q/I$ and $I\subseteq \n^2$. 
We apply Lemma~\ref{RS1.2}, with $a=t-1$, to show that the canonical quotient map $Q\to Q/I=R$ is a Golod homomorphism. 
It follows that $R$ is a Golod ring.   
It suffices to show that 
\begin{enumerate}[\rm(i)]
\item\label{2.2.i} the map $$\Tor^Q_i(R,\kk)\to\Tor^Q_i(R/\mathfrak m^{t-1},\kk),$$ induced by the  quotient map ${R\to R/\mathfrak m^{t-1}}$, is zero for all positive $i$, and 
\item\label{2.2.ii} the map $$\Tor^Q_i(\mathfrak m^{2t-2},\kk)\to\Tor^Q_i(\mathfrak m^{t-1},\kk),$$ induced by the inclusion $\mathfrak m^{2t-2}\to\mathfrak m^{t-1}$ is zero for all non-negative $i$.
 \end{enumerate}

\noindent Condition (\ref{2.2.i}) is established in Lemma~\ref{4.2.a} and
(\ref{2.2.ii}) obviously holds. Indeed, by hypothesis, the top socle degree $s$ of $R$  satisfies $s\le 2t-3$. It  follows that  $\mathfrak m^{2t-2}=0$.
\end{proof}

The following two results are proven in \cite{RS}; but in each case the statement given in \cite{RS} is slightly different than the statement given here.
\begin{setup}\label{setup} Let $(Q,\n,\kk)$ and $(P,\pp,\kk)$ be local rings with $P=Q/(h)$ for some 
element $h$ in  $\n^t$ with $h$ not a zerodivisor on $Q$ and $2\le t$. Let  $N\subseteq M$ be finitely generated $P$-modules, 
 $\incl:N\to M$ represent the inclusion map, and $\varphi:Q\to P$ represent the natural quotient map. For any $P$-module $X$, let $\varphi^X_i:\Tor_i^Q(X,\kk)\to 
\Tor_i^P(X,\kk)$ be the map on $\Tor$ induced by the change of rings $\varphi:Q\to P$. For either ring $A=P$ or $A=Q$, let $\incl_i^A:\Tor_i^A(N,\kk)\to \Tor_i^A(M,\kk)$ be the map on $\Tor$ induced by the $A$-module homomorphism $\incl:N\to M$. \end{setup}

\begin{lemma}\label{6.19}{\rm\cite[Lem.~2.4]{RS}} Adopt the notation of {\rm \ref{setup}}. If 
$\n^{t-1}(M/N)$ is zero, 
then 
$$\ker\left(\varphi_i^M:\Tor_i^Q(M,\kk)\to\Tor_i^P(M,\kk)\right)\subseteq \im\left(\incl_i^Q:\Tor_i^Q(N,\kk)\to\Tor_i^Q(M,\kk)\right)$$for all $i$.\end{lemma}

\begin{lemma}\label{6.20}{\rm\cite[Lem.~2.3]{RS}} Adopt the notation of {\rm\ref{setup}}.
If the modules   $\n^{t-1}N$ and $\n^{t-1}(M/N)$  are both zero, 
then the following statements are equivalent{\rm:}
\begin{enumerate}[\rm(a)]
\item the map $\Tor_i^P(N,\kk)\xrightarrow{\incl_i^P}\Tor_i^P(M,\kk)$ is identically zero for all $i$, and 
\item the composition 
  $\Tor_i^Q(N,\kk)\xrightarrow{\incl_i^Q}\Tor_i^Q(M,\kk)\xrightarrow{\varphi_i^M}\Tor_i^P(M,\kk)$  is identically zero for all $i$.
\end{enumerate}
\end{lemma}
\begin{Remark} To prove these results, in each case start with the short exact sequence
$$0\to N\to M\to M/N\to 0$$ and follow the argument given in \cite{RS}. Keep in mind that the hypothesis that $\n^{t-1}(M/N)$ is zero ensures that 
$$\varphi_i^{M/N}:\Tor_i^Q(M/N,\kk)\xrightarrow{}\Tor_i^P(M/N,\kk)$$ 
is injective for all $i$, see \cite[line 3 on page 427]{RS}.
 In particular, the conclusion we have drawn in Lemma~\ref{6.19} does not require  
$\incl_i^Q:\Tor_i^Q(N,\kk)\to\Tor_i^Q(M,\kk)$ to be the zero map. 

It is worth noting that the change of rings involved in constructing 
$$\varphi_i^M:\Tor_i^Q(M,\kk)\to\Tor_i^P(M,\kk)$$
 is fairly subtle; see \cite[Thm.~3.1.3]{A98} for details. The original construction was due to Shamash \cite{S69}; this construction planted a seed that evolved into the Eisenbud operators.
\end{Remark}

We conclude this section with a  result which exhibits the common denominator for all Poincar\'e series $P^R_M(z)$ when there is a Golod homomorphism from a local hypersurface ring onto $R$ and $M$ roams over all finitely generated $R$-modules.
\begin{theorem}\label{Levin} Let $(Q,\mathfrak n,\kk)$ be a regular local ring
of embedding dimension $e$, $(P,\pp,\kk)$ be a local ring with $P=Q/(h)$ for some $h\in \n^2$, $(R,\m,\kk)$ be a local ring,  $\kappa:P\to R$ be a surjective Golod homomorphism, $\varphi_\bullet^R:\Tor^Q(R,\kk)\to \Tor^P_\bullet(R,\kk)$ be the map induced by the natural quotient map $Q\to P$, and  $d_R(z)$ be the polynomial
$$d_R(z)=1- z(P^Q_R(z)- 1)+(z+z^2)\cdot \big(\HS_{\ker \varphi^R_\bullet}(z)-z\big)\in \mathbb Z[z].$$Then,
for every finitely generated $R$-module $M$, there exists a polynomial $p_M(z)$ in $\mathbb Z[z]$ with
$$P^R_
M(z)d_R(z) = p_M(z).$$
In particular, $p_{\kk}(z)=(1+z)^e$.
 \end{theorem}

\begin{proof} Results of Levin, see for example \cite[Prop.~5.18]{AKM}, give all of the  conclusions, except for the formula for $d_R(z)$. The denominator $d_R(z)$ is calculated in \cite{RS}; although the exact form given above is not explicitly identified there. Most of the steps are well known. One starts with the equation
$$d_R(z) = \frac{(1+z)^e}{P^R_{\kk}(z)};$$ so it suffices to calculate  $P^R_{\kk(z)}$. The homomorphism  $\kappa$ is Golod; hence the equation 
$$P_{\kk}^R(z) = \frac{P^P_{\kk}(z)}
{1 - z(P^P_
R(z) - 1)}$$holds;  see \cite[Prop.~1]{G72}. The key new step is taken in  \cite[2.2.1]{RS} where it is shown that 
\begin{equation}\label{new step}P_X^P(z)=\frac{P^Q_X(z)-(1+z)\cdot \HS_{\ker \varphi^X_\bullet}(z)}{1-z^2},\end{equation} for all finitely generated $P$-modules $X$. (The calculation (\ref{new step}) is valid whenever the hypotheses of \ref{setup} are satisfied.) In the present calculation, one takes $X$ to be $R$.
The ring $P$ 
is a hypersurface; consequently, the Poincar\'e series 
$$P_{\kk}^P(z)=\frac{(1+z)^e}{1-z^2}$$is well known. (Indeed the resolution of $\kk$ by free $P$-modules is known.) Combine everything to obtain the formula for $d_R(z)$.
\end{proof}

\section{Homological consequences of the hypothesis that $R$ is compressed.}\label{AIL}

We deduce three homological consequences of the hypothesis that local Artinian ring $R$ is compressed. These Lemmas (\ref{4.2.b}, \ref{4.2.c}, and \ref{June-19}) play a major role in the proof of the main result, Theorem~\ref{5.1}.

\begin{lemma} \label{4.2.b} Let $(Q,\n,k)$ be a regular local ring and   $(R,\m,k)$ be the local ring $R=Q/I$ for some ideal $I$ of $Q$. Assume that $R$ is 
a compressed 
 local Artinian ring of embedding dimension $e$.
 If ${\tt v}(R)\le \ell$, then the map $\nu_i^Q(\m^\ell)$ of Definition~{\rm\ref{y(1.3.1)}} is zero for $i<e$. 
\end{lemma}

\begin{proof}
Apply Corollary~\ref{I-cor} to see that $\agr{R}$ is 
 a standard-graded compressed   
 Artinian  ${\kk} $-algebra with the top socle degree of $\agr{R}$ equal to the top socle degree of $R$ and ${\tt v}(\agr{R})$ equal to ${\tt v}(R)$; and therefore, 
\cite[Prop. 16]{FL84} guarantees that $\Tor^{\agr{Q}}_i(\agr{R},\kk)$ is concentrated in degrees ${\tt v}(R)+i-1$ and ${\tt v}(R)+i$, for $1\le i\le e-1$. 
Of course, $\Tor^{\agr{Q}}_0(\agr{R},\kk)$ is concentrated in degree $0$. 
Lemma~\ref{Liana!}  ensures that $\nu_i^Q(\m^\ell)$ is identically zero for
all pairs $(i,\ell)$ with 
 $i<e$ and  ${\tt v}(R)\le \ell$. 
\end{proof}

Let $(R,\m,\kk)$ be a local Artinian ring.  This  ring is complete and local; hence the Cohen structure theorem guarantees that $R$ is the quotient of a regular local ring. We often use  information from  Data~\ref{data5}. This information all automatically exists as soon as the local Artinian ring $(R,\m,\kk)$ is chosen. Observe that 
the parameter $t$ of Data~\ref{data5}  is equal to the invariant ${\tt v}(R)$ of \ref{2.2}.(\ref{v(R)}).

 \begin{data}\label{data5}Let $(Q,\n,\kk)$ be a regular local ring and   $(R,\m,\kk)$ be the local Artinian ring ${R=Q/I}$, where  $I$ is an ideal of $Q$ with $I\subseteq \n^2$. Define $t$ to be the largest integer with ${I\subseteq \n^t}$. 
Let $(P,\pp,\kk)$ be the local  hypersurface ring $P=Q/L$, where
$L$ is the principal ideal of $Q$ generated by
a non-zero 
element of   
$I$ which is not in $\mathfrak n^{t+1}$, 
$(K,\partial)$ be the Koszul complex which is a minimal resolution of $\kk$ by free $Q$-modules, and $\pi:Q\to R$ and $\kappa:P\to R$ be the natural quotient homomorphisms. 
\end{data}

\begin{lemma}\label{4.2.c}Let $(R,\m,\kk)$ be  
a compressed 
 local Artinian ring of embedding dimension $e$ and top socle degree $s$.  Adopt Data~{\rm\ref{data5}}. Assume that the field   $\kk$ is infinite  and that $s=2t-1$.
Then there exists $G\in \mathfrak n^{t-1}K_1$ such that $\partial(G)$ generates $L$ and 
$$Z_e(\m^s\otimes_Q K)\subseteq  \bar g Z_{e-1}(\m^t\otimes_QK)\,,$$
where $g$ denotes the image of $G$ in $P\otimes_QK$ and $\bar g$ is the image of $G$ in $R\otimes_QK$. 
\end{lemma}

\begin{proof} 
The field $\kk$ is infinite; therefore we may apply Remark~\ref{Rmk11} and decompose 
$\n$ into subideals $(X_1)+\n'$ with $X_1$ a minimal generator of $\n$,  $\mu(\n')=e-1$,
and 
$h-X_1^t$ in the ideal $\n'\n^{t-1}$ of $Q$, for some
generator $h$ of $L$. 
The decomposition $\n= X_1Q+ \n'$ induces a decomposition $\m=x_1R+\m'$ with $x_1$ equal to the image of $X_1$ and $\m'$ equal to the image of $\n'$. 
Let $\mathfrak q$ be the ideal  $\ann_R(\m')\cap \m^t$ of $R$. 
We proved in Lemma~\ref{FirstStep} that 
\begin{equation}\label{We proved}x_1^{t-1}\mathfrak q=\m^s.\end{equation}
 Let $X_2,\dots,X_e$ be a minimal generating set  for $\n'$ and $T_1,\dots,T_e$ be a basis for $K_1$ with $\partial (T_i)=X_i$.
Recall that $h$ has the property that $h-X_1^t\in (X_2,\dots,X_e)\mathfrak n^{t-1}$.
It follows that there is an  
 element   
$G$  
in 
$K_1$ of the form  \begin{equation}\label{form}G=X_1^{t-1}T_1+\sum_{i=2}^e\alpha_i T_i,\end{equation} for some $\alpha_i\in \mathfrak n^{t-1}$,  with $\partial(G)=h$. The image of $G$ in $R\t_QK$, denoted by $\bar  g$, is a cycle in $Z_1(\mathfrak m^{t-1}\otimes_QK)$. Observe that
\begingroup\allowdisplaybreaks\begin{align*}Z_e(\m^s\otimes_QK)&=\m^s\cdot T_1\cdots T_e\\&=(x_1^{t-1}\mathfrak q)\cdot T_1 \cdots T_e&&\text{by (\ref{We proved})}\\
&=\mathfrak q (x_1^{t-1}T_1)\cdots T_e\\
&=\mathfrak q\Big(\bar g-\sum_{i=2}^e\alpha_i T_i\Big) T_2\cdots  T_e&&\text{by (\ref{form})}\\
&=\mathfrak q \bar g T_2\cdots T_e\\
&=\bar g\mathfrak q T_2\cdots T_e\\
&\subseteq \bar g Z_{e-1}(\mathfrak q \otimes_Q K)&&\text{ since $\partial(T_i)\mathfrak q=0$ 
 for $i\le 2\le e$}. \\
&\subseteq \bar g Z_{e-1}(\mathfrak \m^t \otimes_Q K)&&\text{ since $\mathfrak q\subseteq \m^t$}. 
\end{align*}\endgroup
\vskip-18pt\end{proof}

Lemma~\ref{June-19} is the third of three Lemmas in the section. These Lemmas are used in the proof of the main result. The proof of Lemma~\ref{June-19} is a continuation of the proof of Lemma~\ref{4.2.c}.
\begin{lemma}\label{June-19}
Let $(R,\m,\kk)$ be  
a compressed 
 local Artinian ring of embedding dimension $e$ and top socle degree $s$.
Adopt Data~{\rm\ref{data5}} 
with $s=2t-1$. The following statements hold.
\begin{enumerate}[\rm(a)]\item\label{6.1.b} If $j$ is  an integer which
satisfies $$t+1\le j\le s\quad{and}\quad \socle(R)\cap \m^j=\m^s,$$
then the  maps
$$\nu^P_i:\Tor_i^P(\m^{j},\kk)\to \Tor_i^P(\m^t,\kk),$$induced by the inclusion $\m^j\subseteq \m^t$, 
are zero for all $i$.
\item\label{June-19.b} The maps 
$\Tor_i^R(\m^{s},\kk)\to \Tor_i^R(\m^t,\kk)$, induced by the inclusion $\m^s \subseteq \m^t$, 
are zero for all $i$.
\end{enumerate}
\end{lemma}

\begin{proof} 
 Without loss of generality, we may assume that $\kk$ is infinite. Indeed, 
if $\kk'=\kk(y)$, $Q'=Q[y]_{\mathfrak nQ[y]}$, $P'=P[y]_{\mathfrak pP[y]}$, $R'=R[y]_{\mathfrak mR[y]}$,  and $\mathfrak m'=\mathfrak mR'$,
then   the extensions $Q\to Q'$,  $P\to P'$, and $R\to R'$ are faithfully flat, and therefore, $\nu^P_i=0$ if and only if $\nu^{P'}_i=0$, and  
$$\Tor_i^R(\m^{s},\kk)\to \Tor_i^R(\m^t,\kk)\text{ is zero}\iff
\Tor_i^{R'}({\m'}^{s},\kk')\to \Tor_i^{R'}({\m'}^t,\kk')\text{ is zero}.
$$

\medskip\noindent(\ref{6.1.b}) Let $\nu^Q_i\colon \Tor^Q_i(\mathfrak m^{j},\kk)\to \Tor^Q_i(\mathfrak m^{t},\kk)$ denote the map induced by the inclusion $\mathfrak m^{j}\subseteq\mathfrak m^{t}$. Apply Lemma~\ref{6.20} to the inclusion $\m^{j}\subseteq \m^t$. Observe that $\n^{t-1}$ annihilates $\m^{j}$ and $\m^t/\m^{j}$. Observe also that the map $\incl_i^A$ of  \ref{6.20} is now denoted $\nu_i^A$ for $A=P$ or $A=Q$. We conclude that assertion (\ref{6.1.b}) is equivalent to the assertion
\begin{equation}
\label{=} \varphi_i^{\m^t}\circ \nu^Q_i=0,
\quad \text{for all $i$,}
\end{equation}
where $\varphi^{\m^t}_i\colon \Tor_i^Q(\m^{t}, \kk)\to\Tor_i^P(\m^{t},\kk)$ is the map induced by the natural quotient map $Q\to P$. 
If $0\le i\le e-1$, then Lemma~\ref{4.2.b} yields that the map $\nu_i^Q(\m^t)$ of Definition~\ref{y(1.3.1)} is identically zero. The map $\nu_i^Q$ factors through 
$\nu_i^Q(\m^t)$; therefore, $\nu_i^Q=0$ and (\ref{=}) holds for $0\le i\le e-1$.

We now prove  \eqref{=} for $i=e$.
Recall the Koszul complex $(K,\partial)$ of Data~\ref{data5} which is a resolution of $\kk$ by free  $Q$-modules. We identify the functors
\begin{equation}\label{identify}\HH_{\bullet}(-\t_Q K) \quad\text{and} \quad \Tor_\bullet^Q(-,\kk).\end{equation}
Observe that 
\begin{gather}
 \Tor^Q_e(\mathfrak m^{j},\kk)=\HH_e(\m^{j}\otimes_QK)=\socle(\m^{j})\otimes_QK_e\\\intertext{and}
\Tor^Q_e(\mathfrak m^{t},\kk)=\HH_e(\m^{t}\otimes_QK)=\socle(\m^{t})\otimes_QK_e=(\socle(R)\cap \m^t)\otimes_QK_e.
\end{gather}
The hypothesis that $\socle(R)\cap \m^{j}=\m^s$ yields  $\socle(\m^{j})\otimes_QK_e=\m^s\otimes_QK_e$.
Thus, $\im \nu^Q_e$ is equal to the submodule $\m^s\otimes_QK_e$ of $(\socle(R)\cap \m^t)\otimes_QK_e.$

We
 compute $\varphi^{\m^t}_e(\m^s\otimes_QK_e)$.
 Let $G$ be as in Lemma~\ref{4.2.c}. 
The image of $G$ in $P\t_QK$, denoted by $g$, is a cycle and the minimal resolution of $\kk$ by free $P$-modules is the Tate complex $T=(P\t_QK)\langle Y\rangle$, with \begin{equation}\label{partial(Y)}\partial(Y)=g.\end{equation} 
The homomorphism   $\varphi^{\m^t}_e$  is induced by the natural map 
\begin{equation}\notag\mathfrak m^{t}\otimes_Q K\longrightarrow \mathfrak  m^{t}\otimes_P T=\mathfrak \m^{t}\otimes_P(P\t_Q K)\langle Y\rangle=\mathfrak (\m^{t}\t_Q K)\langle Y\rangle;\end{equation}
hence $\varphi^{\m^t}_e$ is the natural map 
$$
\varphi^{\m^t}_e\colon (\socle(R)\cap \m^t)\otimes_QK_e\to \HH_e\left(\mathfrak (\m^{t}\t_Q K)\langle Y\rangle\right).
$$ Let $z\in \m^s\otimes_QK_e$.  According to Lemma~\ref{4.2.c}, $z=\bar g z'$ for some $z'$ in 
$Z_{e-1}(\m^{t}\t_QK)$,
 where $\bar g$ is the image of $g$ in $R\t_QK$. 
The defining property of $Y$, given in (\ref{partial(Y)}), together with the graded  product rule yields 
\begin{equation}\label{XT}z=\bar g z'=\partial (Y) z'=\partial (Yz')-Y\partial(z')=\partial (Yz'),\end{equation}
which establishes that the image of $z$ under the map $\varphi^{\m^t}_e$ is represented by a boundary
in $(\m^t\t_QK)\langle Y\rangle$; and therefore is zero in $\HH_e\left(\mathfrak (\m^{t}\t_Q K)\langle Y\rangle\right)=\Tor_e^P(\m^t,\kk)$. 
This finishes the proof  of (\ref{=}) and hence the proof of 
(\ref{6.1.b}).

\medskip\noindent(\ref{June-19.b}) Apply Theorem~\ref{Snow} with $b=t$, $\tau=t-1$, $K^R=R\t_Q K$, and $z_1=\bar g$. Recall that  $\bar g\in Z_1(\m^{t-1}\t_QK)$. It is clear that the one-cycle $\bar g$ squares to zero. We verify that hypothesis (\ref{Snow-hyp}) is satisfied.  On the one hand, Lemma~\ref{4.2.c} yields that $$\m^s\t_Q K_e\subseteq \bar g Z_{e-1}(\m^t\t_Q K)$$ and, on the other hand, Lemma~\ref{4.2.b} yields
that $\Tor^Q_i(\m^s\t_Q K)\to \Tor^Q_{i}(\m^{s-1}\t_Q K)$ is the zero map for $i<e$. It follows that $\m^s\t_Q K_i\subset B(\m^{s-1}\t_Q K)$ for $i<e$.
\end{proof}

The following Theorem is a special case of \cite[Thm.~3.1]{Snow}. This result was used in the proof of Lemma~\ref{June-19}.
\begin{theorem}\label{Snow} {\rm \cite{Snow}} Let $(R,\m, \kk)$ be an Artinian local ring with top socle degree $s$, $K^R$ be the Koszul complex on a minimal generating set of $\m$, and   $\tau$ and $b$ be integers
with $s-\tau\le b\le s - 1$ and $2\le \tau+1 \le {\tt v}(R)$.
If there exists a 
cycle $z_1$ in  $Z(\m^\tau K^R)$  with  $z_1^2 = 0$ and  
\begin{equation}\label{Snow-hyp}\m^s K^R \subseteq  z_1\cdot  Z(\m^b K^R)+B(\m^{s-1}K^R),\end{equation} 
then 
the maps $\Tor^R_
i (\m^s, \kk) \to \Tor^R
_i (\m^b, \kk),$ induced by the inclusion $\m^s \subseteq \m^b$, are  zero for all $i$.\end{theorem}

\section{Proof of the main result.}\label{mainResult}

In this section we prove Theorem~\ref{5.1}, which is the main result of the paper.
The short version of the statement is  
``If $R$ is  a compressed local Artinian ring 
with
top socle degree $s$, with $s$ odd, $5\le s$,
and $\socle(R)\cap \m^{s-1}=\m^s$,
then  
the Poincar\'e series of all finitely generated modules over $R$ 
are 
rational, sharing a common denominator,
 and  there is a Golod homomorphism from a complete intersection onto $R$.'' 
Recall that the data of \ref{data5} is constructed from $R$.

\begin{theorem}\label{5.1} Let $(R,\m,\kk)$ be  
a compressed 
 local Artinian ring of embedding dimension $e$ and top socle degree $s$.
 Assume that 
 $s$ is odd, $5\le s$, and
$$\socle(R)\cap \m^{s-1}=\m^s.$$   
Adopt Data~{\rm\ref{data5}}.  Then $s\le 2t-1$ and the following statements hold{\rm:}
$$\begin{cases}
\text{$\kappa:P\to R$ is a Golod homomorphism,}&\text{if $s=2t-1$, and}\\
\text{$\pi:Q\to R$ is a Golod homomorphism,}&\text{if $s<2t-1$}.\end{cases}$$
Furthermore, if $d_R(z)$ is the polynomial
$$d_R(z) = \begin{cases}1 - z
(
P^Q_R
(z) - 1
)
+ c_sz^{e+1}(1 + z),&\text{if $s=2t-1$, and}\\
1 - z
(
P^Q_R
(z) - 1
),&\text{if $s<2t-1$,}\end{cases}$$
where $c_s=\dim_{\kk}(\m^s)$ then,
for every finitely generated $R$-module $M$, there exists a polynomial $p_M(z)$ in $\mathbb Z[z]$ with
$$P^R_
M(z)d_R(z) = p_M(z).$$
In particular, $p_{\kk}(z)=(1+z)^e$.
 \end{theorem}

\begin{proof}
It is shown in Theorem~\ref{comp-ring}.(\ref{xxx.a}) that  $s\le 2t-1$.
If $s<2t-1$, then it is shown in the proof and statement of Observation~\ref{xxx.b} that $\pi$ is a Golod homomorphism and  $R$ is a Golod ring. The statement about the common denominator $d_R(z)$ is due to Lescot \cite{L90}, see also \cite[Thm.~5.3.2]{A98}.

Henceforth, we assume $s=2t-1$. 
The following two conditions hold: 
\begin{chunk-no-advance}\label{6.1.a} the map $\Tor^ P_i(R,\kk)\to\Tor^ P_i(R/\mathfrak m^{t-1},\kk),$ induced by the canonical quotient map $R\to R/\mathfrak m^{t-1}$, is zero for all positive $i$, and\end{chunk-no-advance}
\begin{chunk-no-advance}\label{A6.1.b} 
the map  
$$\nu^P_i:\Tor^P_i(\mathfrak m^{2t-2},\kk)\to \Tor^P_i(\mathfrak m^{t},\kk),$$ 
 induced by the inclusion $\m^{2t-2}\subseteq\mathfrak m^{t}$,
 is zero for all non-negative integers $i$.\end{chunk-no-advance}

Indeed, assertion  \ref{6.1.a} follows from \cite[Lemma 1.4]{RS}, whose proof is similar to the proof of Lemma~\ref{4.2.a} and assertion \ref{A6.1.b} is established in Lemma~\ref{June-19}.(\ref{6.1.b}) with $j=s-1$. The  hypothesis 
$$\socle(R)\cap \m^{s-1}=\m^s$$ of the present result is used to verify the critical hypothesis  $\socle(R)\cap \m^{j}=\m^s$ of Lemma~\ref{June-19}.

Now that  \ref{A6.1.b} holds,  the  map $\Tor^P_i(\mathfrak m^{2t-2},\kk)\to \Tor^P_i(\mathfrak m^{t-1},\kk)$ is also zero, and  Lemma \ref{RS1.2} can be applied with $a=t-1$ to conclude that $\kappa$ is Golod.

Apply Theorem~\ref{Levin} to finish the proof. It remains to prove that the Hilbert series of the kernel of 
$$\varphi^R_\bullet:\Tor_\bullet^Q(R,\kk)\to \Tor_\bullet^P(R,\kk)$$
is $\HS_{\ker(\varphi^R_\bullet)}(z)=z+c_sz^e$. It suffices to prove that
\begin{equation}\label{untagged}\dim_{\kk}\ker(\varphi^R_i)=\begin{cases} 0,&\text{if $i=0$ or $2\le i\le e-1$,}\\
1,&\text{if $i=1$, and}\\
\dim_{\kk}\m^s,&\text{if $i=e$}.\end{cases}\end{equation}
Observe that  $\varphi_0^R:\Tor_0^Q(R,\kk)\to \Tor_0^P(R,\kk)$ is the isomorphism $\kk\to\kk$. It follows that $\dim_{\kk}\ker(\varphi_0^R)=0$. 
Observe that $\varphi_1^R:\Tor_1^Q(R,\kk)\to \Tor_1^P(R,\kk)$ is the natural map
$$\frac{\ker \pi}{\n \ker\pi}\to \frac{\ker\pi}{\n\ker \pi+L}.$$ 
The kernel of this map has  dimension $1$ because one of the minimal generators of $\ker \pi $ has been sent to zero. It is shown in Lemma~\ref{June-16} that  $\ker(\varphi^R_e)\cong \m^s$. We complete the proof of (\ref{untagged}), hence the proof of the Theorem, by showing that 
\begin{equation}\label{from RS}\varphi^R_i\text{ is injective for $2\le i\le e-1$.}\end{equation} Fix $i$ with $2\le i\le e-1$. The hypothesis $$5\le s=2t-1$$ ensures that $3\le t$; hence,
$$\m^{2t-2}\subseteq \m^{t+1}\subseteq \m^t\subseteq \m^{t-1}$$ and 
\begin{equation}\label{J17.1}\Tor^Q_i(\m^{2t-2},\kk)\xrightarrow{\incl_i}\Tor^Q_i(\m^{t-1},\kk)\end{equation} factors through 
\begin{equation}\label{J17.2}\Tor^Q_i(\m^{t+1},\kk)\xrightarrow{\incl_i}\Tor^Q_i(\m^{t},\kk).\end{equation}Lemma~\ref{4.2.b} yields that (\ref{J17.2}) is the zero map; hence,  (\ref{J17.1}) is also the zero map. 
Apply Lemma~\ref{6.19}, together with the fact that (\ref{J17.1}) is the zero map, to the inclusion $\m^{2t-2}\subseteq \m^{t-1}$. Observe that $\n^{t-1}$ annihilates $\m^{t-1}/\m^{2t-2}$. Conclude that $$\varphi_i^{\m^{t-1}}:\Tor_i^Q(\m^{t-1},\kk)\to \Tor_i^P(\m^{t-1},\kk)$$ is injective. 
 One can now employ the commutative diagram in  proof of Claim~2 in the proof of \cite[Lem.~3.4]{RS} to complete the proof of (\ref{from RS}).
\end{proof}

The following calculation is used in the proof of Theorem~\ref{5.1}.
\begin{lemma}\label{June-16}Adopt the notation and hypotheses of Theorem~{\rm\ref{5.1}} with $s=2t-1$. Let $\varphi_e^R\colon \Tor^Q_e(R,k)\to \Tor^P_e(R,k)$ be the map induced by the natural quotient map $Q\to P$ and let $K$ be the Koszul complex which is a minimal resolution of $\kk$ by free $Q$-modules. Then $$\ker(\varphi_e^R)= \m^s\otimes_QK_e.$$ \end{lemma}

\begin{proof} As described at the beginning of the proof of Lemma~\ref{June-19}, it does no harm to assume that $\kk$ is infinite.
The following consequence of Lemma~\ref{4.2.a} is used repeatedly in this proof. 

\begin{chunk-no-advance}\label{repeat} The homomorphism
$\Tor_i^Q(R/\m^t,\kk)\to \Tor^Q_i(R/\m^{t-1},\kk)$, which is induced by the natural quotient map $R/\m^t\to R/\m^{t-1}$, is zero for $1\le i$.\end{chunk-no-advance}

 We continue the identification of the functors
$$\HH_{\bullet}(-\t_Q K) \quad\text{and} \quad \Tor_\bullet^Q(-,\kk)$$ which was begun in (\ref{identify}). In other words, we take $$\Tor^Q_e(R,\kk)\text{ to be }\socle(R)\t_QK_e\quad\text{and}\quad \Tor^P_e(R,k)\text{ to be }\HH_e((R\t_QK)\langle Y\rangle);$$ furthermore, $\varphi_e^R$ carries the cycle $z$ in $\socle(R)\t_QK_e$ to the homology class of $z$ in  $(R\t_QK)\langle Y\rangle$.
The argument (\ref{XT}) shows that if $z\in \m^s\t_QK_e$, then the image of $z$   in $(\m^t\t_QK)\langle Y\rangle$ is a boundary; hence  the image of $z$   in $(R\t_QK)\langle Y\rangle$ is a boundary. Thus, $\m^s\t_QK_e\subseteq \ker (\varphi_e^R)$. We prove the other direction.

Let $w$ be an element of $\socle(R)\t_Q K_e$ which is an element of the kernel of  $\ker\varphi^R_e$. It follows that $w$ is a boundary in $\mathfrak (R\t_Q K)\langle Y\rangle$; therefore,  
\begin{align*}
w&{}=\partial\big(a_0+Ya_1+Y^{(2)}a_2+\dots+Y^{(m)}a_m\big)
\\&{}=\big(\partial(a_0)+\bar ga_{1}\big)+Y\big(\partial(a_1)+\bar ga_{2}\big)+\dots
+Y^{(m-1)}\big(\partial(a_{m-1})+\bar ga_{m}\big)+Y^{(m)}\partial(a_m), \end{align*}
for some  $a_i\in R\t_QK_{e+1-2i}$, with $1\le i\le \lfloor\frac {e+1}2\rfloor$. The module $K_{e+1}$ is zero; consequently,
$a_0=0$. 
The  $(R\t_Q K)$-module $(R\t_Q K)\langle Y\rangle$  is free, with basis $\{Y^{(i)}\}$, and therefore
\begin{align}\label{possible}&w=\bar ga_{1},\quad \partial(a_1)+\bar ga_{2}=0,\quad  \dots,\\
&\partial(a_{m-1})+\bar ga_{m}=0,\quad\text{and}\quad \partial(a_m)=0.\notag\end{align}

It is possible that $m=(e+1)/2$ and $a_m\in R\t_QK_0=R$. Observe that, in this case, $a_m\in \m$. Indeed, if $a_m$ were a unit, then the equation $\partial(a_{m-1})+\bar g a_m=0$ of (\ref{possible}) would yield that $\bar g$ is a boundary in $R\t_QK$ and it would follow from  Lemma~\ref{4.2.c} that $\m^s\t_Q K_e\subseteq \partial (R\t_QK_{e+1})=0$. The most recent statement is impossible because $R$ has top socle degree $s$.

We claim that for each $i$, there exists $b_i\in R\t_QK_{e+2-2i}$, $c_i\in \m^{t-1}\t_QK_{e+1-2i}$, and $d_i\in R\t_QK_{e-2i}$ such that 
\begin{equation}a_i=\partial (b_i)+c_i+\bar gd_i. 
\label{*}\end{equation} We prove (\ref{*}) by descending induction.

If $m<(e+1)/2$, then $a_m$ is a  $(e+1-2m)$-cycle in $R\t_QK$. (Of course, $a_m$ is also a cycle  in $R/\m^t\t_QK$). Apply  (\ref{repeat}) to find $b_m
\in R\t_QK_{e+2-2m}$ and $c_m$ in $\m^{t-1}\t_QK_{e+1-2m}$ with
$a_m=\partial(b_m)+c_m$. 
If $m=(e+1)/2$, then  $a_m\in \m$ and $a_m=\partial(b_m)$ for some $b_m
\in R\t_QK_{1}$.  

In any event, (\ref{*}) holds for  $i=m$.
Suppose, by induction, that (\ref{*}) holds at $i$, for some $i$ with $2\le i\le m$. We will establish (\ref{*}) at $i-1$. 
Apply (\ref{possible}), the induction hypothesis (\ref{*}), the fact that $\bar g$ is a cycle in $R\t_QK$, and the fact that $\bar g\in(R\t_QK)_1$ in order to see that
\begin{align}
0={}&\partial (a_{i-1})+\bar ga_i=
\partial (a_{i-1})+\bar g\left(\partial (b_i)+c_i+\bar gd_i\right)\notag\\
={}&\partial \big(a_{i-1}- (\bar gb_{i})
\big)+\bar gc_i.
\label{at any rate}\end{align}
The product $\bar gc_i$ is in $\m^{2t-2}\t K_{e+2-2i}$; and therefore,  
equation (\ref{at any rate}) exhibits 
$a_{i-1}- (\bar gb_{i})
$ as a cycle in $R/\m^t\t_QK$. 
Apply (\ref{repeat}) to find $b_{i-1}$ in $R\t_QK_{e+4-2i}$ and $c_{i-1}\in \m^{t-1}\t_QK_{e+3-2i}$
with $$a_{i-1}- (\bar gb_{i})
=\partial(b_{i-1})+c_{i-1}.
$$ Thus, (\ref{*}) holds at $i-1$. 

By induction, (\ref{*}) holds at $i=1$ and 
\begin{align*}w&{}=\bar ga_1= \bar g\big(\partial (b_1)+c_1+\bar gd_1
\big)\\&{}
= -\partial (\bar gb_1)+\bar gc_1
=\bar gc_1
,\end{align*}
for some
$b_1\in R\t_QK_{e}$, $c_1\in \m^{t-1}\t_QK_{e-1}$, and $d_1\in R\t_QK_{e-2}$. 
We used the fact that  $\bar gb_1\in R\t_QK_{e+1}=0$.
 Thus,
$$w=\bar gc_1\in (\socle(R)\cap \m^{2t-2})\t_QK_e.$$ We have assumed that $s=2t-1$ and that $$\socle(R)\cap \m^{s-1}=\m^s.$$ It follows that $w\in \m^s\t_Q K_e$, and the proof is complete.
\end{proof}

\section{Factoring out the highest power of the maximal ideal.}\label{consolation}
The hypotheses $s=2t-1$, $5\le s$, and $\socle(R)\cap \m^{s-1}=\m^s$ all are in effect in the interesting case of the main theorem, Theorem~\ref{5.1}. If we only assume $s=2t-1$, then we are not able to make any claim about the Poincar\'e series $P^R_{\kk}$; nonetheless, in Corollary~\ref{6.3}, we 
prove that the homomorphism $R\to R/\m^s$ is Golod. 
As a consequence, when  all of the hypotheses of the interesting case of Theorem~\ref{5.1} are reimposed, 
 we 
 prove, 
in Corollary~\ref{6.3.5},
that  $R/\m^s$ is a Golod ring.

\begin{corollary}\label{6.3} Let $(R,\m,\kk)$ be a compressed local Artinian ring with 
top socle degree $s$. If  $s=2{\tt v}(R)-1$, then  
the natural quotient homomorphism $\rho:R\to R/\m^s$ is Golod.
\end{corollary} 

\begin{proof} Let $t={\tt v}(R)$. 
 It is shown in Lemma~\ref{June-19}.(\ref{June-19.b}) that the maps 
$$\Tor_i^R(\m^{s},\kk)\to \Tor_i^R(\m^t,\kk),$$ induced by the inclusion $\m^s \subseteq \m^t$, 
are zero for all $i$. 
It follows that the maps 
\begin{equation}\label{(*)}\Tor_i^R(R/\m^{s},\kk)\to \Tor_i^R(R/\m^t,\kk),\end{equation} induced by the natural quotient homomorphism  $R/\m^s \to R/\m^t$,  are zero for all positive $i$.
Apply Lemma \ref{RS1.2} with $P=R$,  $R$ replaced by $R/\m^s$, and $a=t$.  Condition (\ref{C4.2.a}) of  Lemma~\ref{RS1.2} is satisfied by  (\ref{(*)}). Condition (\ref{C4.2.b}) of  Lemma~\ref{RS1.2}  holds because $\m^{2t}=0$. 
Conclude that $\rho$ is a Golod homomorphism.
\end{proof}
 
The next result describes how to use a mapping cone to obtain a minimal resolution of the $Q$-module $R/\m^s$ if one already knows the minimal resolution of $R$. 

\begin{lemma}\label{c1}Let $(R,\m,\kk)$ be a compressed local Artinian ring of
 embedding dimension $e$ and   top socle degree $s$, $(Q,\n,\kk)$ be a regular local ring of embedding dimension $e$ with $R=Q/I$ for some ideal $I$ of $Q$,  and  $c_s$ be  $\dim_{\kk}\m^s$. If ${\tt v}(R)+1\le s$, then  $$P^Q_{R/\m^s}(z)=P^Q_R(z)+c_sz(1+z)^e-c_sz^e(1+z).$$\end{lemma}

\begin{proof} Observe that  the inclusion $\m^s\subseteq R$ induces  the
following statements:
\begin{equation}\label{Tor}\begin{cases}
\Tor_i^Q(\m^s,\kk)\to \Tor_i^Q(R,\kk)\text{ is zero }&\text{for $0\le i\le e-1$, and }\\
\Tor_i^Q(\m^s,\kk)\to \Tor_i^Q(R,\kk)\text{ is an injection}&\text{for $i=e$.}
\end{cases}\end{equation}
Before  establishing (\ref{Tor}); we draw consequences from these statements. One combines (\ref{Tor}) and   the short exact sequence
\begin{equation} 
\label{ses}
 0\to \m^s\to R\to R/\m^s\to0\end{equation}  to relate the Betti numbers (denoted $b_i(M)$) of the $Q$-modules $M=R/\m^s$, $M=R$, and   $M=\m^s$. Keep in mind that  $\m^s$ is isomorphic to the direct sum of $c_s$ copies of $\kk$. The Betti numbers are related by 
$$b_i(R/\m^s)=\begin{cases} 
b_0(R),&\text{if $i=0$},\\
b_i(R)+c_sb_{i-1}(\kk),&\text{if $1\le i\le e-1$, and}\\
b_e(R)-c_s+c_sb_{e-1}(\kk),&\text{if $i=e$}.\end{cases}$$
It follows that 
$$P_{R/\m^s}^Q(z)=P_{R}^Q(z)+c_sP^Q_{\kk}(z)-c_sz^e-c_sz^{e+1},$$ as claimed.

Now we prove (\ref{Tor}). The long exact sequence of homology that is associated to (\ref{ses}) ends with
$$0\to\Tor^Q_e(\m^s,\kk)\to  \Tor^Q_e(R,\kk);$$ hence, the lower line in (\ref{Tor}) holds. On the other hand, if $0\le i\le e-1$, then  Lemma~\ref{4.2.b} guarantees that the inclusion $\m^{\ell+1}\subseteq \m^\ell$ induces the zero map
$$\Tor^Q_i(\m^{\ell+1},\kk)\to \Tor^Q_i(\m^{\ell},\kk)$$ for all $\ell$ with ${\tt v}(R)\le \ell$. The hypothesis ensures  that ${\tt v}(R)\le s-1$; hence,
$$\Tor^Q_i(\m^{s},\kk)\to \Tor^Q_i(\m^{s-1},\kk)$$ is the zero map for  $i<e$. The top line of (\ref{Tor}) holds because the inclusion
 $\m^s\subseteq R$ factors through the inclusion $\m^s\subseteq \m^{s-1}$.
\end{proof}

In the interesting case of the main theorem, the ring $R/\m^s$ is Golod.
\begin{corollary}\label{6.3.5} Let $(R,\m,\kk)$ be a compressed local Artinian ring with  
top socle degree $s$. If $s=2{\tt v}(R)-1$, 
$5\le s$, and $\socle(R)\cap \m^{s-1}=\m^s$, then the ring $R/\m^s$ is Golod. 
\end{corollary}
\begin{proof}
 Let $e$ be the embedding dimension of $R$, $c_s$ be $\dim_{\kk}\m^s$, and $(Q,\n,\kk)$ be a regular local ring with $R=Q/I$ for some ideal $I\subseteq \n^2$. Recall (see, for example, \cite[(5.0.1)]{A98} or \ref{2.14}) that
$R/\m^s$ is Golod if and only if 
\begin{equation}\label{c2}P_{\kk}^{R/\m^s}(z)=\frac{P^Q_{\kk}(z)}{1-z(P^Q_{R/\m^s}(z)-1)}.
\end{equation}  We calculate both sides of (\ref{c2}), verify the equality, and thereby prove the result.
Observe first that 
\begin{equation}\label{c3}P^R_{R/\m^s}(z)=1+c_szP^R_{\kk}(z).
\end{equation}
Indeed, the exact sequence 
$$0\to \m^s\to R\to R/\m^s\to 0$$ is the beginning of  the minimal resolution of $R/\m^s$ by free $R$-modules and $\m^s$ is isomorphic to $\bigoplus_{c_s}\kk$.

The hypotheses of Theorem~\ref{5.1} are in effect; and therefore,
\begin{equation}\label{c4} P_{\kk}^R(z)=\frac{(1+z)^e}{1 - z
(
P^Q_R
(z) - 1
)
+ c_sz^{e+1}(1 + z)}.\end{equation}
The map $R\to R/\m^s$ is Golod by Corollary~\ref{6.3}; and therefore,
\begin{equation}\label{c5}P_{\kk}^{R/\m^s}(z)=\frac{P^R_{\kk}(z)}{1-z(P^R_{R/\m^s}(z)-1)};\end{equation}see, for example \cite[Prop. 1]{G72} or \cite[Prop.~3.3.2]{A98}. Use (\ref{c5}), (\ref{c3}), (\ref{c4}), and then Lemma~\ref{c1} to calculate
\begingroup\allowdisplaybreaks\begin{align*}P_k^{R/\m^s}(z)&{}=\frac{P^R_{\kk}(z)}{1-z(P^R_{R/\m^s}(z)-1)}=\frac{P^R_{\kk}(z)}{1-c_sz^2P^R_{\kk}(z)}\\
&{}=\frac{1}{(P^R_{\kk}(z))^{-1}-c_sz^2}\\
&{}=\frac{1}{
\frac{1 - z
(
P^Q_R
(z) - 1
)
+ c_sz^{e+1}(1 + z)}{(1+z)^e}-c_sz^2}\\
&{}=\frac{(1+z)^e}{
1 - z(P^Q_R(z) - 1)+ c_sz^{e+1}(1 + z)-c_sz^2(1+z)^e}\\
&{}=\frac{(1+z)^e}{
1 - z\big(P^Q_R(z)+c_sz(1+z)^e -c_sz^{e}(1 + z)- 1\big)}\\
&{}=\frac{(1+z)^e}{
1 - z\big(P^Q_{R/\m^s}(z)- 1\big)}.\end{align*}\endgroup
Apply (\ref{c2}) to conclude   that 
$R/\m^s$ is a Golod ring.
\end{proof}

\medskip\noindent{\bf Acknowledgment.} An earlier version of this paper studied compressed level local Artinian rings.
Lars Christensen encouraged the authors to remove the hypothesis ``level'' and they  appreciate this suggestion.


\begin{thebibliography}{99}
\bibitem{An2} D.~Anick, {\em A counterexample to a conjecture of Serre}, Ann. of Math. {\bf 115} (1982),  1--33;  Ann. of Math.  {\bf 116} (1982),  661. 
\bibitem{Av86}L.~Avramov, {\it
Golod homomorphisms}, Algebra, algebraic topology and their interactions (Stockholm, 1983), 59--78,
Lecture Notes in Math., {\bf 1183}, Springer, Berlin, 1986. 

\bibitem{Av94}L.~Avramov, {\em Local rings over which all modules have rational Poincar\'e series}, J. Pure Appl. Algebra {\bf 91} (1994),  29--48. 


\bibitem{A98}L.~Avramov, {\em Infinite free resolutions} Six lectures on commutative algebra (Bellaterra, 1996), 1--118, Progr. Math., {\bf 166}, Birkh\"auser, Basel, 1998.

\bibitem{AINS} L.~Avramov, S.~Iyengar,  S.~Nasseh, and S.~Sather-Wagstaff,
{\em Homology over trivial extensions of commutative DG algebras}, 
{\tt https://arxiv.org/abs/1508.00748}


\bibitem{AISS}L.~Avramov, S.~Iyengar, S.~Nasseh, and   S.~Sather-Wagstaff,  {\em Persistence of homology over commutative Noetherian rings}, work in progress, 2016.

\bibitem{AKM}L.~Avramov, A.~Kustin, and M.~Miller,  {\em
Poincar\'e series of modules over local rings of small embedding codepth or small linking number},
J. Algebra {\bf 118} (1988),  162--204. 

\bibitem{Ba} I.~Babenko, {\em
Problems of growth and rationality in algebra and topology}, 
Uspekhi Mat. Nauk  {\bf 41} (1986), 95--142; Russian Math. Surveys {\bf 41} (1986), 117--176. 

\bibitem{B92}   M.~Beintema, 
{\em Gorenstein algebras with unimodal $h$-sequences},
Comm. Algebra {\bf 20} (1992),  979--997.


\bibitem{B83} R.~B{\o}gvad, {\em Gorenstein rings with transcendental Poincar\'e series}, Math. Scand. {\bf 53} (1983),  5--15. 

\bibitem{B99}M.~Boij, {\em Betti numbers of compressed level algebras}, J. Pure Appl. Algebra {\bf 134}, (1999), 111--131. 

\bibitem{B00}   M.~Boij, 
{\em Artin level modules},
J. Algebra {\bf 226} (2000),  361--374.


\bibitem{BL}M.~Boij and D.~Laksov, {\em 
Nonunimodality of graded Gorenstein Artin algebras},
Proc. Amer. Math. Soc. {\bf 120} (1994),  1083--1092.


\bibitem{CENR} G.~Casnati, J.~Elias, R.~Notari, and M.~E.~Rossi,  {\em Poincar\'e series and deformations of Gorenstein local algebras}, Comm. Algebra {\bf 41} (2013),  1049--1059. 


\bibitem{CN}G.~Casnati and R.~Notari, {\em The Poincar\'e series of a local Gorenstein ring of multiplicity up to {\rm10} is rational}, Proc. Indian Acad. Sci. Math. Sci. {\bf 119} (2009),  459--468. 

\bibitem{C10}   B.~Coolen, 
{\em Resolutions of some Gorenstein algebras with nonunimodal Hilbert functions},
Comm. Algebra {\bf 38} (2010), 1595--1612.



\bibitem{Snow}  A.~Croll, R.~Dellaca, A.~Gupta, J.~Hoffmeier, V.~Mukundan, D.~ Rangel Tracy, L.~M.~\c Sega, G.~Sosa, and P.~Thompson,
{\em Detecting Koszulness and related homological properties from the algebra structure of Koszul homology}, {\tt https://arxiv.org/abs/1611.04001}.

\bibitem{D14}   A.~De Stefani, 
{\em Artinian level algebras of low socle degree}. 
Comm. Algebra {\bf 42} (2014),  729--754.


\bibitem{ER15} J.~Elias and M.~E.~Rossi, {\em
Analytic isomorphisms of compressed local algebras}, 
Proc. Amer. Math. Soc. {\bf 143} (2015),  973--987.


\bibitem{EV}J.~Elias and G.~Valla, {\em A family of local rings with rational Poincar\'e series}, Proc. Amer. Math. Soc. {\bf 137} (2009),  1175--1178. 


\bibitem{F87}   R.~Fr\"oberg, 
{\em Connections between a local ring and its associated graded ring},
J. Algebra {\bf 111} (1987),  300--305.

\bibitem{FL84}  R.~Fr\"oberg and D.~Laksov,  {\em Compressed Algebras},  Conference on Complete Intersections in Acireale,  Lecture Notes in Mathematics,   {\bf 1092} (1984), Springer-Verlag, New York,  121--151.


\bibitem{M2}D.~Grayson and M.~Stillman,
          {\em Macaulay2, a software system for research in algebraic geometry},
           {Available at \tt{http://www.math.uiuc.edu/Macaulay2/}}

\bibitem{GP98} M.~Gross and S.~Popescu, {\em Equations of $(1,d)$-polarized abelian surfaces}, Math. Ann. {\bf 310} (1998),  333--377. 


\bibitem{G72}T.~Gulliksen, {\em
Massey operations and the Poincar\'e series of certain local rings},
J. Algebra {\bf 22} (1972), 223--232.


\bibitem{HR}J.~Herzog and G.~Restuccia, {\em Regularity functions for homogeneous algebras}, Arch. Math. (Basel) {\bf 76} (2001),  100--108. 

\bibitem{HS17} J.~Hoffmeier and L.~M.~\c Sega, 
{\em Conditions for the Yoneda algebra of a local ring to be generated in low degrees}, 
J. Pure Appl. Algebra {\bf 221} (2017),  304--315.



\bibitem{I84}  A.~Iarrobino  {\em Compressed algebras{\rm:} Artin algebras having given socle degrees and maximal length}, Trans. Amer. Math. Soc. {\bf 285}  (1984),  337--378.

\bibitem{K07}  J.~Kleppe, 
{\em Families of Artinian and one-dimensional algebras},
J. Algebra {\bf 311} (2007),  665--701.



\bibitem{L90}J.~ Lescot, {\em S\'eries de Poincar\'e et modules inertes},  J. Algebra {\bf 132} (1990),  22--49.


\bibitem{L16}J.~Lutz, {\em Homological characterizations of quasi-complete intersections}, preprint, 2016.



\bibitem{MP} J.~McCullough and   I.~Peeva, {\em Infinite  Graded  Free  Resolutions}, Commutative Algebra and Noncommutative Algebraic Geometry. I. 
MSRI Publications, {\bf 67} (2015), 215--258.



\bibitem{MM03}   J.~Migliore and R.~Mir\'o-Roig,
{\em On the minimal free resolution of $n+1$ general forms}, 
Trans. Amer. Math. Soc. {\bf 355} (2003),  1--36.



\bibitem{MMN05}   J.~Migliore, R.~Mir\'o-Roig, and U.~Nagel,
{\em Minimal resolution of relatively compressed level algebras},
J. Algebra {\bf 284} (2005),  333--370.



\bibitem{Ro}J.-E.~Roos, {\em
Homology of loop spaces and of local rings}, {\rm18}th Scandinavian Congress of Mathematicians (Aarhus, 1980), pp. 441--468,
Progr. Math., {\bf11}, Birkh\"auser, Boston, Mass., 1981. 

\bibitem{Ro05}J.-E.~Roos, {\em
Good and bad Koszul algebras and their Hochschild homology}, 
J. Pure Appl. Algebra {\bf 201} (2005), 295--327. 

\bibitem{RS} M.~E.~Rossi and  L.~M.~\c{S}ega, {\em
Poincar\'e series of modules over compressed Gorenstein local rings}, 
Adv. Math. {\bf 259} (2014), 421--447. 

\bibitem{S80}P.~Schenzel, {\em \"Uber die freien Aufl\"osungen extremaler Cohen-Macaulay-Ringe}, 
J. Algebra {\bf 64} (1980),  93--101. 

\bibitem{Se65} J.-P. Serre, {\em Alg\`ebre Locale, Multiplicit\'es}, Lecture Notes in Mathematics {\bf 11}, Springer-Verlag,
Berlin, 1965.

\bibitem{S69}J.~Shamash, {\em The Poincar\'e series of a local ring}, J. Algebra {\bf 12} (1969), 453--470.

 
\bibitem{Z03}   
F.~Zanello, {\em
Extending the idea of compressed algebra to arbitrary socle-vectors},
J. Algebra {\bf 270} (2003),  181--198.

\bibitem{Z04}       F. Zanello, {\em Extending the idea of compressed algebra to arbitrary socle-vectors, II. Cases of non-existence}, J. Algebra {\bf 275} (2004),  732--750. 



\end{thebibliography}
\end{document}